\documentclass[12pt]{amsart}

\usepackage{amssymb}

\usepackage{enumerate}

\usepackage{graphicx}

\makeatletter
\@namedef{subjclassname@2010}{%
  \textup{2010} Mathematics Subject Classification}
\makeatother

\newtheorem{thm}{Theorem}[section]
\newtheorem{cor}[thm]{Corollary}
\newtheorem{lem}[thm]{Lemma}
\newtheorem{prop}[thm]{Proposition}

\theoremstyle{definition}
\newtheorem{defin}[thm]{Definition}

\numberwithin{equation}{section}

\newcommand{\id}{{\rm{id}}}
\newcommand{\Ad}{{\rm{Ad}}}
\newcommand{\Hom}{{\rm{Hom}}}

\newcommand{\BN}{\mathbf N}
\newcommand{\BC}{\mathbf C}

\newcommand{\la}{\langle}
\newcommand{\ra}{\rangle}

\newcommand{\lInd}{{\rm{l-Ind}}}
\newcommand{\rInd}{{\rm{r-Ind}}}
\newcommand{\End}{{\rm{End}}}

\begin{document}


\baselineskip=17pt



\title[The strong Morita equivalence for coactions]
{The strong Morita equivalence for coactions of a finite dimensional $C^*$-Hopf algebra on unital $C^*$-algebras}
\author{Kazunori Kodaka}
\address{Department of Mathematical Sciences, Faculty of Science, Ryukyu
University, Nishihara-cho, Okinawa, 903-0213, Japan}
\email{kodaka@math.u-ryukyu.ac.jp}

\author{Tamotsu Teruya}
\address{Faculty of Education, Gunma University, 4-2 Aramaki-machi,
Maebashi City, Gunma, 371-8510, Japan}
\email{teruya@gunma-u.ac.jp}

\date{}

\begin{abstract}
Following Jansen and Waldmann, and Kajiwara and Wata-tani,
we shall introduce notions of coactions of a finite dimensional $C^*$-Hopf algebra
on a Hilbert $C^*$-bimodule of finite type in the sense of Kajiwara and Watatani
and define their crossed product. We shall investigate their basic properties and
show that the strong Morita equivalence for coactions preserves the Rohlin property
for coactions of a finite dimensional $C^*$-Hopf algebra on unital $C^*$-algebras.
\end{abstract}



\maketitle

\section{Introduction}\label{sec:intro}
Let $A$ and $B$ be unital $C^*$-algebras and $X$ a Hilbert $A-B$-bimodule of
finite type in the sense of Kajiwara and Watatani \cite {KW1:bimodule}.
Let $H$ be a finite dimensional $C^*$-Hopf algebra with its dual $C^*$-Hopf
algebra $H^0$. In this paper, following Jansen and Waldmann \cite {JW:covariant},
we shall introduce the notion of coactions of $H^0$ on $X$
and define their crossed product. That is, for coactions $\rho$ and $\sigma$ of $H^0$
on $A$ and $B$, respectively, we introduce a linear map $\lambda$ from $X$ to $X\otimes H^0$, which
is compatible with the coactions $\rho$, $\sigma$ and the left $A$-module action,
the right $B$-module action and the left $A$-valued and right $B$-valued inner products.
Then we can define the crossed product $X\rtimes_{\lambda}H$, which is a Hilbert $A\rtimes_{\rho}H-
B\rtimes_{\sigma}H$-bimodule of finite type. Furthermore, we shall give a duality theorem similar to
the ordinary one. This theorem in the case of countably discrete group actions and Kac systems
are found in Kajiwara and Watatani \cite {KW2:discrete} and Guo and Zhang \cite {GZ:Kac}, respectively.
The latter result is almost a generalization of our duality theorem. But our approach to coactions of a finite
dimensional $C^*$-Hopf algebra on a unital $C^*$-algebra is a useful addition to our paper,
especially the main result on preservation of the Rohlin property under the strongly Morita equivalence.
So, in Section \ref{sec:duality}, we give a duality theorem, a version of crossed product duality for coactions
of finite dimensional $C^*$-Hopf algebras on Hilbert $C^*$-bimodules of finite type.
Also, we see that if $X$ is an $A-B$-equivalence bimodule, we can show that $X\rtimes_{\lambda}H$
is an $A\rtimes_{\rho}H-B\rtimes_{\sigma}H$-equivalence bimodule. Hence $A\rtimes_{\rho}H$ is
strongly Morita equivalent to $B\rtimes_{\sigma}H$. Finally, if $X$ is an $A-B$-equivalence bimodule and
$\rho$ has the Rohlin property, then $\sigma$ has also the Rohlin property.
As an application of the result, we can obtain the following: Under a certain condition, if a unital $C^*$-algebra $A$ has a finite
dimensional $C^*$-Hopf algebra coaction of $H^0$ with the Rohlin property, then any unital $C^*$-algebra
that is strongly Morita equivalent to $A$ has also a finite dimensional $C^*$-Hopf algebra
coaction of $H^0$ with the Rohlin property. In \cite [Section 4]{OKT:finite}, we gave an incorrect example
of an action of a finite dimensional $C^*$-Hopf algebra on a unital $C^*$-algebra with the Rohlin property.
But applying the above result to  \cite [Section 7]{KT2:coaction}, we can give several examples of them.
\par
For an algebra $A$, we denote by $1_A$ and $\id_A$ the unit element in $A$ and the identity map on
$A$, respectively. If no confusion arises, we denote them by $1$ and $\id$, respectively. For each
$n\in \BN$, we denote by $M_n (\BC)$ the $n\times n$-matrix algebra over $\BC$ and $I_n$ denotes the
unit element in $M_n (\BC)$.
\par
For projections $p, q$ in a $C^*$-algebra $A$, we write $p\sim q$ in $A$ if $p$ and $q$ are Murray-von
Neumann equivalent in $A$.

\section{Preliminaries}\label{sec:pre}Let $H$ be a finite dimensional $C^*$-Hopf algebra.
We denote its comultiplication, counit and antipode by $\Delta$, $\epsilon$ and $S$, respectively. We shall
use Sweedler's notation $\Delta(h)=h_{(1)}\otimes h_{(2)}$ for any $h\in H$
which suppresses a possible summation when we write comultiplications. We denote by $N$
the dimension of $H$. Let $H^0$ be the dual $C^*$-Hopf algebra of $H$. We denote its
comultiplication, counit and antipode by $\Delta^0$, $\epsilon^0$ and $S^0$,
respectively. There is the distinguished
projection $e$ in $H$. We note that $e$ is the Haar trace on $H^0$. Also, there is the distinguished projection
$\tau$ in $H^0$ which is the Haar trace on $H$. Since $H$ is finite dimensional,
$H\cong\oplus_{k=1}^L M_{f_k }(\BC)$ and $H^0 \cong\oplus_{k=1}^K M_{d_k }(\BC)$
as $C^*$-algebras. Let $\{\, v_{ij}^k \, | \, k=1,2,\dots,L, \, i,j=1,2,\dots,f_k \, \}$
be a system of matrix units of $H$. Let $\{\, w_{ij}^k \, |\, k=1,2,\dots,K, \, i,j=1,2,\dots,d_k \}$
be a basis of $H$ satisfying Szyma\'nski and Peligrad's \cite [Theorem 2.2,2]{SP:saturated},
which is called  a system of {\it comatrix} units of $H$, that is, the dual basis of a system
of matrix units of $H^0$. Also let $\{\, \phi_{ij}^k \, |\, k=1,2,\dots,K, \, i,j=1,2,\dots,d_k \}$
and $\{\, \omega_{ij}^k \, |\, k=1,2,\dots,L, \, i,j=1,2,\dots,f_k \}$ be systems of matrix units and comatrix units of $H^0$,
respectively.
\par
Let $A$ and $B$ be unital $C^*$-algebras and $X$ a Hilbert $A-B$-bimodule of finite type in the sense of
\cite {KW1:bimodule}. We regard a $C^*$-Hopf algebra $H^0$ as an $H^0 -H^0$-equivalence bimodule in the usual way.
\par
Let $X\otimes H^0$ be an exterior tensor product of Hilbert $C^*$-bimodules $X$ and $H^0$, which is a
Hilbert $A\otimes H^0 -B\otimes H^0$-bimodule.

\begin{lem}\label{lem:tensor}With the above notations, $X\otimes H^0$
is a Hilbert $A\otimes H^0 -B\otimes H^0$-bimodule of finite type. In particular,
if $X$ is an $A-B$-equivalence bimodule, $X\otimes H^0$ is an $A\otimes H^0 -B\otimes H^0$-
equivalence bimodule.
\end{lem}
\begin{proof}Since $X$ is of finite type, there is a right $B$-basis $\{u_i \}_{i=1}^n $
of $X$. Then for any $x\in X$, $\phi\in H^0$,
$$
\sum_{i=1}^n (u_i \otimes 1^0 )\la u_i \otimes 1^0 \, , \, x\otimes\phi \ra_{B\otimes H^0}
=\sum_{i=1}^n (u_i \otimes 1^0 )(\la u_i \, , \, x \ra_B \otimes\phi ) \\
=x\otimes\phi .
$$
Thus a family $\{ u_i \otimes 1^0 \}_{i=1}^n $ is a right $B\otimes H^0$-basis of $X\otimes H^0$.
In the same way as above, we can see that there is a left $A\otimes H^0$-basis of $X\otimes H^0$. Hence by
\cite [Proposition 1.12]{KW1:bimodule} or \cite [Lemma 1.3]{KW2:discrete}, $X\otimes H^0$ is
a Hilbert $A\otimes H^0 -B\otimes H^0$-bimodule of finite type. Furthermore,
we suppose that $X$ is an $A-B$-equivalence bimodule. Since $X$ is full with the both-sided inner
products, by the definitions of the
left and right inner products of $X\otimes H^0$, so is $X\otimes H^0$. Moreover, the associativity
condition of the left and right inner products of $X\otimes H^0$ holds since the associativity condition
of the left and right inner products of $X$ holds. Hence $X\otimes H^0$ is an $A\otimes H^0 -B\otimes H^0$-
equivalence bimodule.
\end{proof}

Let $\Hom (H, X)$ be the vector space of all linear maps from $H$ to $X$. Then $X\otimes H^0$ is
isomorphic to $\Hom (H, X)$ as vector spaces. Sometimes, we identify $X\otimes H^0$ with $\Hom (H, X)$.

\section{Coactions of a finite dimensional $C^*$-Hopf algebra on a Hilbert $C^*$-bimodule of finite type and
strong Morita equivalence}
\label{sec:coaction}
Let $A$ and $B$ be unital $C^*$ algebras and $X$ a Hilbert $A-B$-bimodule of finite type. Let
$H$ be a finite dimensional $C^*$-Hopf algebra with its dual $C^*$-Hopf algebra $H^0$.
Let $\rho$ be a weak coaction of $H^0$ on $A$ and $\lambda$ a linear map
from $X$ to $X\otimes H^0$. Following \cite{JW:covariant}, \cite{KW2:discrete},
we shall introduce the several definitions.

\begin{defin}\label{Def:wcoactionleft}
With the above notations, we say that $(A, X, \rho, \lambda, H^0 )$ is a {\it weak left covariant system}
if the following conditions hold:
\newline
(1) $\lambda(ax)=\rho(a)\lambda(x)$ for any $a\in A$, $x\in X$,
\newline
(2) $\rho({}_A \la x \, , \,y \ra)={}_{A\otimes H^0}\la \lambda(x) \, , \, \lambda(y) \ra$ for any
$x, y\in X$,
\newline
(3) $(\id_X \otimes\epsilon^0 )\circ\lambda =\id_X$.

We call $\lambda$ a {\it weak left coaction} of $H^0$ on $X$ with respect to $(A, \rho)$.
\end{defin}

We define the weak action of $H$ on $A$ induced by $\rho$ as follows: For any $a\in A$, $h\in H$,
$$
h\cdot_{\rho}a=(\id\otimes h)(\rho(a)),
$$
where we regard $H$ as the dual space of $H^0$. In the same way as above, we can define the action
of $H$ on $X$ induced by $\lambda$ as follows: For any $x\in X$, $h\in H$,
$$
h\cdot_{\lambda}x=(\id\otimes h)(\lambda(x))=\lambda(x)^{\widehat{}}(h),
$$
where $\lambda(x)^{\widehat{}}$ is the element in $\Hom(H, X)$ induced by $\lambda(x)$ in
$X\otimes H^0$. Then we obtain the following conditions which are equivalent to Conditions
(1)-(3) in Definition \ref{Def:wcoactionleft}, respectively:
\newline
(1)' $h\cdot_{\lambda}ax=[h_{(1)}\cdot_{\rho}a][h_{(2)}\cdot_{\lambda}x]$ for any $a\in A$, $x\in X$, $h\in H$,
\newline
(2)' $h\cdot_{\rho}{}_A \la x ,y \ra ={}_A \la [h_{(1)}\cdot_{\lambda}x] , [S(h_{(2)}^* )\cdot_{\lambda}y] \ra$
for any $x,y\in X$, $h\in H$,
\newline
(3)' $1_H \cdot_{\lambda}x=x$ for any $x\in X$.
\par
Let $\sigma$ be a weak coaction of $H^0$ on $B$.

\begin{defin}\label{Def:wcoactionright}With the above notations, we say that $(B, X, \sigma, \lambda, H^0 )$ is a
{\it weak right covariant system} if the following conditions hold:
\newline
(4) $\lambda(xb)=\lambda(x)\sigma(b)$ for any $b\in B$, $x\in X$,
\newline
(5) $\sigma(\la x , y \ra_B )=\la\lambda(x) , \lambda(y) \ra_{B\otimes H^0}$ for any $x, y\in X$,
\newline
(6) $(\id_X \otimes\epsilon^0 )\circ\lambda=\id_X$.
\end{defin}
We call $\lambda$ a
{\it weak right coaction}
of $H^0$ on $X$ with respect to $(B, \sigma )$.
We can also define the weak action of $H$ on $X$ induced by $\lambda$ satisfying the similar
conditions to (1)'-(3)'. That is, we have the following conditions which are equivalent to Conditions (4)-(6), respectively:
\newline
(4)' $h\cdot_{\lambda}xb=[h_{(1)}\cdot_{\lambda}x][h_{(2)}\cdot_{\sigma}b]$ for any $b\in B$, $x\in X$, $h\in H$,
\newline
(5)' $h\cdot_{\sigma}\la x , y \ra_B =\la [S(h_{(1)}^* )\cdot_{\lambda}x] \, , \, [h_{(2)}\cdot_{\lambda}y ] \ra_B$ 
for any $x,y\in X$, $h\in H$,
\newline
(6)' $1_H \cdot_{\lambda}x=x$ for any $x\in X$.
\par
Let $\rho$ and $\sigma$ be weak coactions of $H^0$ on $A$ and $B$, respectively. Let
$X$ be a Hilbert $A-B$-bimodule of finite type.

\begin{defin}\label{Def:wcovariant}We say that $(A, B, X,  \rho, \sigma, \lambda, H^0 )$ is a
{\it weak covariant system} if the following conditions hold:
\newline
(1) $\lambda(ax)=\rho(a)\lambda(x)$ for any $a\in A$, $x\in X$,
\newline
(2) $\lambda(xb)=\lambda(x)\sigma(b)$ for any $b\in B$, $x\in X$,
\newline
(3) $\rho({}_A \la x \, , \,y \ra)={}_{A\otimes H^0}\la \lambda(x) \, , \, \lambda(y) \ra$ for any
$x, y\in X$,
\newline
(4) $\sigma(\la x , y \ra_B )=\la\lambda(x) , \lambda(y) \ra_{B\otimes H^0}$ for any $x, y\in X$,
\newline
(5) $(\id_X \otimes\epsilon^0 )\circ\lambda =\id_X$.
\end{defin}
We call $\lambda$ a {\it weak coaction} of $H^0$ on $X$ with respect to $(A, B, \rho, \sigma )$.
We note that the above conditions are equivalent to the following conditions, respectively:
\newline
(1)' $h\cdot_{\lambda}ax=[h_{(1)}\cdot_{\rho}a][h_{(2)}\cdot_{\lambda}x]$ for any $a\in A$, $x\in X$, $h\in H$,
\newline
(2)' $h\cdot_{\lambda}xb=[h_{(1)}\cdot_{\lambda}x][h_{(2)}\cdot_{\sigma}b]$ for any $b\in B$, $x\in X$, $h\in H$,
\newline
(3)' $h\cdot_{\rho}{}_A \la x ,y \ra ={}_A \la [h_{(1)}\cdot_{\lambda}x] , [S(h_{(2)}^* )\cdot_{\lambda}y] \ra$
for any $x,y\in X$, $h\in H$,
\newline
(4)' $h\cdot_{\sigma}\la x , y \ra_B =\la [S(h_{(1)}^* )\cdot_{\lambda}x] \, , \, [h_{(2)}\cdot_{\lambda}y ] \ra_B$ 
for any $x,y\in X$, $h\in H$,
\newline
(5)' $1_H \cdot_{\lambda}x=x$ for any $x\in X$.
\vskip 0.2cm
We extend the above notions
to coactions of a finite dimensional $C^*$-Hopf algebra on unital $C^*$-algebras.

\begin{defin}\label{Def:coaction}Let $A, B$ and $H, H^0$ be as above. Let $\rho$
and $\sigma$ be coactions of $H^0$ on $A$ and $B$, respectively and
let $X$ be a Hilbert $A-B$-bimodule of finite type.
\newline
(i) We say that $(A, X, \rho, \lambda, H^0 )$ is a
\it
left covariant system
\rm
if it is a weak left covariant system and a weak left coaction $\lambda$ of $H^0$ on $X$ with
respect to $(A, \rho)$ satisfies that
\newline
($*$) $(\lambda\otimes\id)\circ\lambda=(\id\otimes\Delta^0 )\circ\lambda$,
\newline
which is equivalent to the condition that
\newline
($*$)' $h\cdot_{\lambda}[l\cdot_{\lambda}x]=hl\cdot_{\lambda}x$
for any $x\in X$ , $h, l\in H$.
\par
We call $\lambda$ a
\sl
left coaction
\rm
of $H^0$ on $X$ with respect to $(A, \rho)$.
\newline
(ii) We say that $(B, X, \sigma, \lambda, H^0 )$ is a
\it
right covariant system
\rm
if it is a weak right covariant system and a weak right coaction $\lambda$ of $H^0$ on $X$ with
respect to $(B, \sigma)$ satisfies ($*$) or ($*$)'.
We call $\lambda$ a
\sl
right coaction
\rm
of $H^0$ on $X$
with respect to $(B, \sigma )$.
\newline
(iii) We say that $(A, B, X, \rho, \sigma, \lambda, H^0 )$ is a
{\it covariant system} if it is a weak covariant system and a weak
coaction $\lambda$ with respect to $(A, B, \rho, \sigma )$ satisfies ($*$) or ($*$)'.
We call $\lambda$ a {\it coaction} of $H^0$ on $X$ with respect to $(A, B, \rho, \sigma)$.
\end{defin}

Furthermore, we extend the notion of the covariant system to twisted coactions
of a finite dimensional $C^*$-Hopf algebra on unital $C^*$-algebras.
We recall the definition of a twisted coaction $(\rho, u)$ of a $C^*$-Hopf algebra
$H^0$ on a unital $C^*$-algebra $A$ (See [9], [10]). Let $\rho$ be a weak coaction of
$H^0$ on $A$ and $u$ a unitary element in $A\otimes H^0 \otimes H^0$.
Then we say that $(\rho, u)$ is a {\sl twisted coaction} of $H^0$ on $A$ if
the following conditions hold:
\newline
(1) $(\rho\otimes\id)\circ\rho =\Ad(u)\circ(\id\otimes\Delta^0 )\circ\rho$,
\newline
(2) $(u\otimes 1^0 )(\id\otimes\Delta^0 \otimes\id)(u)=(\rho\otimes\id\otimes\id)(u)(\id\otimes\id\otimes\Delta^0 )(u)$,
\newline
(3) $(\id\otimes h\otimes\epsilon^0 )(u)=(\id\otimes\epsilon^0 \otimes h)(u)=\epsilon^0 (h)1^0$ for any $h\in H$.
\newline
The above conditions are equivalent to the following conditions, respectively:
\newline
(1)' $h\cdot_{\rho}[l\cdot_{\rho}a]=\widehat{u}(h_{(1)}, l_{(1)})[h_{(2)}l_{(2)}\cdot_{\rho}a]\widehat{u^*}(h_{(3)}, l_{(3)})$, for
any $a\in A$, $h, l\in H$,
\newline
(2)' $\widehat{u}(h_{(1)}, l_{(1)})\widehat{u}(h_{(2)}l_{(2)}, m)
=[h_{(1)}\cdot_{\rho}\widehat{u}(l_{(1)}, m_{(1)})]\widehat{u}(h_{(2)}, l_{(2)}m_{(2)})$, for any $h, l, m\in H$,
\newline
(3)' $\widehat{u}(h, 1)=\widehat{u}(1, h)=\epsilon(h)1^0 $ for any $h\in H$.

\begin{defin}\label{Def:twisted}Let $A, B$ and $H, H^0$ be as above. Let $(\rho, u)$
and $(\sigma, v)$ be twisted coactions of $H^0$ on $A$ and $B$, respectively and
let $X$ be a Hilbert $A-B$-bimodule of finite type. We say that $(A, B, X, \rho, u, \sigma, v, \lambda, H^0 )$ is
a {\it twisted covariant system} if it is a weak covariant system and a weak coaction $\lambda$
of $H^0$ with respect to $(A, B, \rho, \sigma)$ satisfies that
\newline
($* *$) $(\lambda\otimes\id)(\lambda(x)) =u(\id\otimes\Delta^0 )(\lambda(x))v^* $ for any $x\in X$, 
\newline
which is equivalent to the condition that
\newline
($**$)' $h\cdot_{\lambda}[l\cdot_{\lambda}x]
=\widehat{u}(h_{(1)}, l_{(1)})[h_{(2)}l_{(2)}\cdot_{\lambda}x]\widehat{v}^* (h_{(3)}, l_{(3)})$
for any $x\in X$, $h, l\in H$,
\newline
where $\widehat{u}$ and $\widehat{v}$ are elements in $\Hom (H\times H, A)$ and $\Hom (H\times H, B)$
induced by $u\in A\otimes H^0 \otimes H^0$ and $v\in B\otimes H^0 \otimes H^0$, respectively.
We call $\lambda$ a {\it twisted coaction} of $H^0$ on $X$ with respect to $(A, B, \rho, u, \sigma, v )$.
\end{defin}
Next, we introduce the notion of strong Morita equivalence for coactions of a finite dimensional
$C^*$-Hopf algebra on unital $C^*$-algebras.

\begin{defin}\label{Def:morita}Let $A, B$ and $H, H^0$ be as above.
\newline
(i) Let $\rho$ and $\sigma$ be weak coactions of $H^0$ on $A$ and $B$, respectively.
We say that $\rho$ is {\it strongly Morita equivalent} to $\sigma$ if there are an $A-B$-equivalence
bimodule $X$ and a weak coaction $\lambda$ of $H^0$ on $X$ such that $(A, B, X, \rho, \sigma, \lambda, H^0 )$
is a weak covariant system.
\newline
(ii) Let $\rho$ and $\sigma$ be coactions of $H^0$ on $A$ and $B$, respectively.
We say that $\rho$ is {\it strongly Morita equivalent} to $\sigma$ if there are an $A-B$-equivalence
bimodule $X$ and a coaction $\lambda$ of $H^0$ on $X$ such that $(A, B, X, \rho, \sigma, \lambda, H^0 )$
is a covariant system.
\newline
(iii) Let $(\rho, u)$ and $(\sigma, v)$ be twisted coactions of $H^0$ on $A$ and $B$, respectively.
We say that $(\rho, u)$ is {\it strongly Morita equivalent} to $(\sigma, v)$ if there are an $A-B$-equivalence
bimodule $X$ and a twisted coaction $\lambda$ of $H^0$ on $X$ such that $(A, B, X, \rho, u, \sigma, v, \lambda, H^0 )$
is a twisted covariant system.
\end{defin}

We shall show that the above strong Morita equivalences are equivalence
relations.
\begin{prop}\label{prop:relation}The strong Morita equivalence of weak coactions of a finite
dimensional $C^*$-Hopf algebra on a unital $C^*$-algebra is an equivalence relation.
\end{prop}
\begin{proof}It suffices to show the transitivity since the other conditions clearly hold.
Let $A, B, C$ be unital $C^*$-algebras and let $X$ and $Y$ be an $A-B$-equivalence bimodule
and a $B-C$-equivalence bimodule,
respectively. Let $\rho$, $\sigma$ and $\gamma$ be weak coactions of $H^0$ on $A, B$ and $C$,
respectively. We suppose that $\rho$ is strongly Morita equivalent to $\sigma$ and that
$\sigma$ is strongly Morita equivalent to $\gamma$. Let $\lambda$ and $\mu$ be weak coactions
of $H^0$ on $X$ and $Y$ satisfying Definition \ref{Def:morita}(i), respectively.
Then $X\otimes_B Y$ is an $A-C$-equivalence bimodule.
We define a bilinear map `` $\cdot_{\lambda\otimes\mu}$'' from $H\times (X\otimes_B Y)$ to
$X\otimes_B Y$ as follows: For any $x\in X$, $y\in Y$, $h\in H$,
$$
h\cdot_{\lambda\otimes\mu}(x\otimes y)=[h_{(1)}\cdot_{\lambda}x]\otimes [h_{(2)}\cdot_{\mu}y] .
$$
Then we can show that the above map `` $\cdot_{\lambda\otimes\mu}$" satisfies
Conditions (1)'-(5)' in Definition \ref{Def:wcovariant} by routine computations.
\end{proof}

\begin{cor}\label{cor:relation}The strong Morita equivalence of twisted coactions of a finite
dimensional $C^*$-Hopf algebra on a unital $C^*$-algebra is an equivalence relation.
\end{cor}
\begin{proof}By Proposition \ref{prop:relation}, we have only to show that Condition ($**$)' in
Definition \ref{Def:twisted}. Let $(\rho, u)$, $(\sigma, v)$ and $(\gamma, w)$ be twisted coactions
of $H^0$ on unital $C^*$-algebras $A$, $B$ and $C$, respectively. Let the other notations be as in the proof of
Proposition \ref{prop:relation}.
For any $x\in X$, $y\in Y$, $h, l\in H$,
\begin{align*}
& h\cdot_{\lambda\otimes\mu}[l\cdot_{\lambda\otimes\mu}x\otimes y]
=[h_{(1)}\cdot_{\lambda}
[l_{(1)}\cdot_{\lambda}x]]\otimes[h_{(2)}\cdot_{\mu}[l_{(2)}\cdot_{\mu}y]] \\
& =\widehat{u}(h_{(1)}, l_{(1)})[h_{(2)}l_{(2)}\cdot_{\lambda}x]\otimes [h_{(3)}l_{(3)}\cdot_{\mu}y]
\widehat{w}^* (h_{(4)}, l_{(4)}) \\
& =\widehat{u}(h_{(1)}, l_{(1)})[h_{(2)}l_{(2)}\cdot_{\lambda\otimes\mu}(x\otimes y)]\widehat{w}^* (h_{(3)}, l_{(3)}).
\end{align*}
Therefore, we obtain the conclusion.
\end{proof}

Of course, the notion of the strong Morita equivalence of coactions of a finite dimensional $C^*$-Hopf algebra on unital
$C^*$-algebras is an extension of that of actions of a finite group on unital $C^*$-algebras.
We shall show it.
Let $G$ be a finite group and $\alpha$ an action of $G$ on a unital $C^*$-algebra $A$.
We consider the coaction of $C(G)$ on $A$ induced by the action $\alpha$ of $G$ on $A$.
We denote it by the same symbol $\alpha$. That is,
$$
\alpha: A\longrightarrow A\otimes C(G), \quad a\longmapsto \sum_{t\in G}\alpha_t (a)\otimes\delta_t
$$
for any $a\in A$, where for any $t\in G$, $\delta_t $ is a projection in $C(G)$ defined by
$$
\delta_t (s)=\begin{cases} 0 & \text{if $s\ne t$} \\
1 & \text{if $s=t$} \end{cases}.
$$
Let $B$ be a unital $C^*$-algebra and $\beta$ an action of $G$
on $B$. We denote by the same symbol $\beta$ the coaction of $C(G)$ on $B$
induced by the action $\beta$.

\begin{prop}\label{prop:extension}With the above notations, the following conditions are equivalent:
\newline
$(1)$ The actions $\alpha$ and $\beta$ of $G$ on $A$ and $B$ are strongly Morita equivalent,
\newline
$(2)$ The coactions $\alpha$ and $\beta$ of $C(G)$ on $A$ and $B$ are strongly Morita equivalent.
\end{prop}
\begin{proof}We suppose Condition (1). Then by Raeburn and Williams \cite [Definition 7.2]{RW:continuous}, there are
an $A-B$-equivalence bimodule $X$ and an action $u$ of $G$ by linear isomorphisms of $X$
such that
$$
\alpha_t ( \, {}_A \la x, y \ra)={}_A \la u_t (x) \, , \, u_t (y) \ra , \quad
\beta_t ( \la x, y \ra_B )=\la u_t (x) \, , \, u_t (y) \ra_B 
$$
for any $x, y \in X$, $t\in G$. We note that by \cite [Remark 7.3]{RW:continuous}, we have the following equations:
$$
u_t (ax)=\alpha_t (a)u_t (x), \quad u_t (xb)=u_t (x)\beta_t (b)
$$
for any $a\in A$, $b\in B$, $x\in X$, $t\in G$. Let $\lambda$ be a linear map from $X$ to $X\otimes C(G)$ defined by
for any $x\in X$,
$$
\lambda(x)=\sum_{t\in G}u_t (x)\otimes \delta_t .
$$
Then by routine computations, we can see that $\lambda$ is a coaction of $C(G)$ on $X$ with respect to
$(A, B, \alpha, \beta)$.
Hence we obtain Condition (2). Next we suppose Condition (2). Then there are an $A-B$-equivalence
bimodule $X$ and a coaction $\lambda$ of $C(G)$ on $X$ with respect to
$(A, B, \alpha, \beta)$.
We regard $G$ as a subset of $C^* (G)$. For any $t\in G$, we define a linear map $u_t$ on $X$ as
follows: For any $x\in X$, $u_t(x)=t\cdot_{\lambda}x$. Then for any $t, s\in G$, $x\in X$,
$$
u_t (u_s (x))=t\cdot_{\lambda}[s\cdot_{\lambda}x]=ts\cdot_{\lambda}x=u_{ts}(x).
$$
Thus we can see that $u$ is an action of $G$ by linear isomorphisms of $X$, which satisfies the desired
conditions by easy computations. Thus we obtain (1).
\end{proof}
Modifying \cite [Example 7.4(b)]{RW:continuous}, we shall obtain the following lemma, which can give
examples of the strong Morita equivalence of
coactions of a finite dimensional $C^*$-Hopf algebra on a unital $C^*$-algebra.
Before it, we introduce the following definition:

\begin{defin}\label{Def:wexterior}Let $\rho$ and $\sigma$ be weak coactions of $H^0$ on $A$.
We say that $\rho$ is {\it exterior equivalent} to $\sigma$ if
there is a unitary element $w\in A\otimes H^0$ satisfying that
$$
\sigma=\Ad(w)\circ\rho, \quad (\id\otimes\epsilon^0 )(w)=1.
$$
\end{defin}

\begin{lem}\label{lem:wexterior}Let $\rho$ and $\sigma$ be weak coactions of $H^0$ on $A$.
Then the following conditions are equivalent:
\newline
$(1)$ The weak coactions $\rho$ and $\sigma$ are exterior equivalent, 
\newline
$(2)$ The weak coactions $\rho$ and $\sigma$ are strongly Morita equivalent by a weak coaction
$\lambda$ from an $A-A$-equivalence bimodule ${}_A A_A$ to
an $A\otimes H^0 -A\otimes H^0$-equivalence bimodule ${}_{A\otimes H^0}A\otimes {H^0}_{A\otimes H^0}$,
where we regard $A$ and $A\otimes H^0$ as an $A-A$-equivalence bimodule and
an $A\otimes H^0-A\otimes H^0$-equivalence bimodule in the usual way, respectively.
\end{lem}
\begin{proof}We suppose Condition (1). Then there is a unitary element $w\in A\otimes H^0$ satisfing that
$\sigma=\Ad(w)\circ\rho$ and that $(\id\otimes\epsilon^0 )(w)=1$. Let $\lambda$ be a linear map from
${}_A A_A$ to ${}_{A\otimes H^0}A\otimes {H^0}_{A\otimes H^0}$
defined by $\lambda(x)=\rho(x)w^*$ for any $x\in {}_A A_A$. 
By routine computations, we can see that $\lambda$ is a weak coaction of $H^0$ on ${}_A A_A$
with respect to $(A, A, \rho, \sigma )$. Next, we suppose Condition (2). 
We note that $\lambda$ is a weak coaction of $H^0$ on ${}_A A_A$ with respect to
$(A, A, \rho, \sigma )$. We identify $A\otimes H^0$ with $\End_{A\otimes H^0}(A\otimes {H^0}_{A\otimes H^0})$,
where $\End_{A\otimes H^0}(A\otimes {H^0}_{A\otimes H^0})$ is a $C^*$-algebra
of all right $A\otimes H^0$-module maps on
$A\otimes {H^0}_{A\otimes H^0}$. Let $w=\theta_{\lambda(1)^* ,1\otimes 1^0}$ be a rank-one
operator on $A\otimes {H^0}_{A\otimes H^0}$ induced by $\lambda(1)^*$ and $1\otimes 1^0$.
Then $w$ is a unitary element in $\End_{A\otimes H^0}(A\otimes {H^0}_{A\otimes H^0})$.
Indeed, for any $x\in A\otimes {H^0}_{A\otimes H^0}$
\begin{align*}
ww^* (x) & =\lambda(1)^* (1\otimes 1^0 )\lambda(1)x =\la \lambda (1) \, , \, \lambda(1) \ra_{A\otimes H^0} \,x
=\sigma(1)x=x, \\
w^* w(x) & =\lambda(1)\lambda(1)^* x
={}_{A\otimes H^0} \la \lambda(1) \, , \, \lambda(1) \ra x=\rho(1)x=x .
\end{align*}
Also, for any $a\in A$, $x\in A\otimes {H^0}_{A\otimes H^0}$
\begin{align*}
(w\rho(a)w^* )(x) & =w(\rho(a)\lambda(1)x)
=\lambda(1)^* \lambda(a)x=\la \lambda(1) \, , \, \lambda(a) \ra_{A\otimes H^0}\, x \\
& =\sigma(a)x .
\end{align*}
Thus $w$ is a unitary element in $A\otimes H^0$ and $\sigma=\Ad(w)\circ\rho$.
Furthermore, let $z=(\id\otimes\epsilon^0 )(w)$. Then $z$ is a unitary element in $A$ such that
$az=za$ for any $a\in A$ since $\sigma =\Ad(w)\circ\rho$. Let $w_1 =w(z^* \otimes 1^0 )$. Then
$w_1$ is a unitary element in $A\otimes H^0$ satisfying that $\sigma=\Ad(w_1 )\circ\rho$ and
that $(\id\otimes\epsilon^0 )(w_1 )=1$.
Therefore we obtain Condition (1).
\end{proof}

\begin{lem}\label{lem:exterior}Let $(\rho, u)$ and $(\sigma, v)$ be twisted coactions of $H^0$ on $A$.
Then the following conditions are equivalent:
\newline
$(1)$ The twisted coactions $(\rho, u)$ and $(\sigma, v)$ are exterior equivalent, 
\newline
$(2)$ The twisted coactions $(\rho, u)$ and $(\sigma, v)$ are strongly Morita equivalent by a twisted coaction
$\lambda$ from an $A-A$-equivalence bimodule ${}_A A_A$ to
an $A\otimes H^0 -A\otimes H^0$-equivalence bimodule ${}_{A\otimes H^0}A\otimes {H^0}_{A\otimes H^0}$,
where we regard $A$ and $A\otimes H^0$ as an $A-A$-equivalence bimodule and
an $A\otimes H^0-A\otimes H^0$-equivalence bimodule in the usual way, respectively.
\end{lem}
\begin{proof}We suppose Condition (1). Then there is a unitary element $w\in A\otimes H^0$ such that
$$
\sigma =Ad(w)\circ\rho \, , \quad v=(w\otimes 1^0 )(\rho\otimes\id)(w)u(\id\otimes\Delta^0 )(w^* ) .
$$
Let $\lambda$ be as in the proof of Lemma \ref{lem:wexterior}.
Then for any $x\in{}_A A_A$,
\begin{align*}
((\lambda\otimes\id)\circ\lambda)(x) & =u(\id\otimes\Delta^0)(\rho(x))u^* (\rho\otimes\id)(w^* )(w^* \otimes1^0 ) \\
& =u(\id\otimes\Delta^0 )(\rho(x))(\id\otimes\Delta^0 )(w^* )v^* \\
& =u(\id\otimes\Delta^0 )(\lambda(x))v^* .
\end{align*}
Thus by Lemma \ref{lem:wexterior}, $\lambda$ is a twisted coaction of $H^0$ on ${}_A  A_A$ with respect to
$(A, A, \rho, u, \sigma, v)$.
Next, we suppose Condition (2). We note that $\lambda$ is a twisted coaction of $H^0$ on ${}_A A_A$ with respect
to $(A, A, \rho, u, \sigma, v)$.
We identify $A\otimes H^0$ with $\End_{A\otimes H^0}(A\otimes {H^0}_{A\otimes H^0})$,
where $\End_{A\otimes H^0}(A\otimes {H^0}_{A\otimes H^0})$ is a $C^*$-algebra
of all right $A\otimes H^0$-module maps on
$A\otimes {H^0}_{A\otimes H^0}$. Let $w=\theta_{\lambda(1)^* ,1\otimes 1^0}$ be a rank-one
operator on $A\otimes {H^0}_{A\otimes H^0}$ induced by $\lambda(1)^*$ and $1\otimes 1^0$.
Then $w$ is a unitary element in $\End_{A\otimes H^0}(A\otimes {H^0}_{A\otimes H^0})$
such that $\sigma=\Ad(w)\circ\rho$ by Lemma \ref{lem:wexterior}.
We note that $w^* ={}_{A\otimes H^0} \la \lambda(1) \, , \, 1\otimes 1^0 \ra$.
Indeed, for any $x\in A\otimes {H^0}_{A\otimes H^0}$
$$
w^* x =(1\otimes 1^0 )\la \lambda(1)^* \,  , \, x \ra_{A\otimes H^0}
=\lambda(1)x
={}_{A\otimes H^0} \la \lambda(1) \, , \, 1\otimes 1^0 \ra x .
$$
Hence $w^* ={}_{A\otimes H^0} \la \lambda(1) \, , \, 1\otimes 1^0 \ra$.
Thus
\begin{align*}
(\rho\otimes\id)(w^* ) & =(\rho\otimes\id)({}_{A\otimes H^0} \la \lambda(1) \, , \, 1\otimes 1^0 \ra ) \\
& ={}_{A\otimes H^0 \otimes H^0} \la ((\lambda\otimes\id)\circ\lambda)(1) \, , \, \lambda(1)\otimes 1^0 \ra \\
& ={}_{A\otimes H^0 \otimes H^0} \la u((\id\otimes\Delta^0 )\circ\lambda)(1)v^* \, , \, \lambda(1)\otimes 1^0 \ra \\
& =u((\id\otimes\Delta^0 )\circ\lambda)(1)v^* (\lambda(1)^* \otimes 1^0 ).
\end{align*}
It follows that
\begin{align*}
& (\rho\otimes\id)(w^* )(w^* \otimes 1^0 ) \\
& =u((\id\otimes\Delta^0 )\circ\lambda)(1)v^* (\lambda(1)^* \otimes 1^0 )
({}_{A\otimes H^0} \la \lambda(1) \, , \, 1\otimes 1^0 \ra\otimes 1^0 ) \\
& =u((\id\otimes\Delta^0 )\circ\lambda)(1)v^* (\la \lambda(1) \, , \,  \lambda(1) \ra_{A\otimes H^0}\otimes 1^0 ) \\
& =u((\id\otimes\Delta^0 )\circ\lambda)(1)v^* (\sigma(\la 1 \, , \, 1 \ra_A )\otimes 1^0  ) \\
& =u(\id\otimes\Delta^0 )(\lambda(1))v^* =u(\id\otimes\Delta^0 )({}_{A\otimes H^0 }
\la \lambda(1) \, , \, 1\otimes 1^0 \ra)v^* \\
& =u(\id\otimes\Delta^0 )(w^* )v^* .
\end{align*}
Thus $v=(w\otimes 1^0 )(\rho\otimes\id)(w)u(\id\otimes\Delta^0 )(w^* )$.
Therefore we obtain Condition (1).
\end{proof}

Next, we discuss on relations between innerness, outerness and strong Morita equivalence.
Let $\rho_{H^0}^A$ be the trivial coaction of $H^0$ on $A$.

\begin{lem}\label{lem:inner}$(i)$ Let $\rho$ be a weak coaction of $H^0$ on $A$. Then the following
conditions are equivalent:
\newline
$(1)$ The weak coaction $\rho$ is inner,
\newline
$(2)$ The weak coaction $\rho$ is strongly Morita equivalent to $\rho_{H^0}^A$.
\newline
$(ii)$ Let $\rho$ be a coaction of $H^0$ on $A$. Then the following
conditions are equivalent:
\newline
$(1)$ The coaction $\rho$ is strongly inner,
\newline
$(2)$ The coaction $\rho$ is strongly Morita equivalent to $\rho_{H^0}^A$.
\end{lem}
\begin{proof}(i) We suppose that $\rho$ is inner. Then we can see that there is a unitary element $w\in A\otimes H^0$
satisfying that $\rho=\Ad(w)\circ\rho_{H^0}^A$ and that $(\id\otimes\epsilon^0 )(w)=1$ in the same way as in the
proof that Condition (2) implies Condition (1) in Lemma \ref{lem:wexterior}.
Thus $\rho$ is exterior equivalent to $\rho_{H^0}^A$.
Hence by Lemma \ref{lem:wexterior}, $\rho$ is strongly Morita equivalent to $\rho_{H^0}^A$.
Next we suppose that $\rho$ is strongly Morita equivalent to $\rho_{H^0}^A$. Then there are
an $A-A$-equivalence bimodule $X$ and a weak coaction $\lambda$ of $H^0$ on $X$ with respect to
$(A, A, \rho, \rho_{H^0}^A )$. We note that for any $a\in A$, $x\in X$,
$$
\lambda(xa)=\lambda(x)\rho_{H^0}^A (a)=\lambda(x)(a\otimes 1^0 ).
$$
For any $h\in H$, let $\widehat{w}(h)$ be a linear map on $X$ defined by for any $x\in X$,
$$
\widehat{w}(h)x=h\cdot_{\lambda}x .
$$
Then by the above discussion, $\widehat{w}(h)\in \End_A (X)$, where $\End_A (X)$ is a $C^*$-algebra
of all right $A$-module maps on $X$. Since $X$ is an $A-A$-equivalence bimodule, we can identify
$\End_A (X)$ with $A$ and we can regard $\widehat{w}(h)$ as an element in $A$ for any $h\in H$.
Furthermore, since the map $h\mapsto\widehat{w}(h)$ is linear, $\widehat{w}\in\Hom(H, A)$.
Let $w$ be the element in $A\otimes H^0$ induced by $\widehat{w}\in\Hom(H, A)$. By the definition
of $w$, clearly $\widehat{w}(1)=1$.
We show that $w$ is a unitary element in $A\otimes H^0$ such that $\rho=\Ad(w)\circ\rho_{H^0}^A$.
For any $x, y\in X$, $h\in H$,
\begin{align*}
& \la (\widehat{w^* }\widehat{w})(h)x \, , \, y\ra_A =\la \widehat{w}(h_{(2)})x \, , \, \widehat{w}(S(h_{(1)}^* ))y \ra_A \\
& =\la [h_{(2)}\cdot_{\lambda}x] \, ,\, [S(h_{(1)}^* )\cdot_{\lambda}y] \ra_A
=S(h^* )\cdot_{\rho_{H^0}^A} \la x , y \ra_A =\la \epsilon(h)x \, , \, y \ra_A .
\end{align*}
Thus $w^* w=1\otimes 1^0$.
Also, for any $x, y\in X$, $h\in H$,
\begin{align*}
h\cdot_{\rho} {}_A \la x, y \ra & ={}_A \la [h_{(1)}\cdot_{\lambda}x] \, , \, [S(h_{(2)}^* )\cdot_{\lambda}y] \ra
= {}_A \la \widehat{w}(h_{(1)})x \, , \, \widehat{w}(S(h_{(2)}^* ))y \ra \\
& =\widehat{w}(h_{(1)}){}_A \la x , y \ra\widehat{w^* }(h_{(2)}) .
\end{align*}
Hence $\rho=\Ad(w)\circ\rho_{H^0}^A$ since $X$ is an $A-A$-equivalence bimodule. Thus
$ww^* =w\rho_{H^0}^A (1)w^* =\rho(1)=1\otimes 1^0$.
Therefore, the weak coaction $\rho$ is inner.
\newline
(ii) We suppose that a coaction $\rho$ is strongly inner. Then $\rho$ is exterior equivalent to $\rho_{H^0}^A$.
Hence by Lemma \ref{lem:exterior}, $\rho$ is strongly Morita equivalent to $\rho_{H^0}^A$.
Next, we suppose that $\rho$ is strongly Morita equivalent to $\rho_{H^0}^A$.
Then there are
an $A-A$-equivalence bimodule $X$ and a coaction $\lambda$ of $H^0$ on $X$ with respect to
$(A, A, \rho, \rho_{H^0}^A )$. Let $w$ be as in (i). It suffices to show that for any $h, l\in H$,
$\widehat{w}(hl)=\widehat{w}(h)\widehat{w}(l)$. For any $x\in X$, $h, l\in H$,
$$
\widehat{w}(h)\widehat{w}(l)x=h\cdot_{\lambda}[l\cdot_{\lambda}x]=hl\cdot_{\lambda}x=\widehat{w}(hl).
$$
Therefore, $\rho$ is strongly inner.
\end{proof}

Let $\rho_{H^0}^A$ and $\rho_{H^0}^B$ be the trivial coactions of $H^0$ on $A$ and $B$, respectively.
We suppose that $A$ and $B$ are strongly Morita equivalent and let $X$ be an $A-B$-equivalence bimodule.
Then $\rho_{H^0}^A$ and $\rho_{H^0}^B$ are strongly Morita equivalent. If a linear map $\lambda_{H^0}^X$
from $X$ to $X\otimes H^0$ is defined by $\lambda_{H^0}^X (x)=x\otimes 1^0$ for any $x\in X$,
then the $\lambda_{H^0}^X$ is a coaction of $H^0$ on $X$ with respect to $(A, B, \rho_{H^0}^A , \rho_{H^0}^B )$.

\begin{cor}\label{cor:inner}$(i)$ Let $\rho$ and $\sigma$ be weak coactions of $H^0$ on $A$ and $B$, respectively. If
$\rho$ is strongly Morita equivalent to $\sigma$, then
$\rho$ is inner if and only if so is $\sigma$.
\newline
$(ii)$ Let $\rho$ and $\sigma$ be coactions of $H^0$ on $A$ and $B$, respectively. If $\rho$ is strongly Morita equivalent to
$\sigma$, then
$\rho$ is strongly inner if and only if so is $\sigma$.
\end{cor}
\begin{proof}(i) We suppose that $\rho$ is inner. Then
$\sigma$ is strongly Morita equivalent to $\rho_{H^0}^B$ by Lemma \ref{lem:inner}(i), Proposition
\ref{prop:relation} and the above discussion. Therefore, $\sigma$ is inner by Lemma \ref{lem:inner}(i).
\newline
(ii) We suppose that $\rho$ is strongly inner. Then $\sigma$ is strongly Morita equivalent to $\rho_{H^0}^B$
by Lemma \ref{lem:inner}(ii), Corollary \ref{cor:relation} and the above discussion. Therefore, $\sigma$ is
strongly inner by Lemma \ref{lem:inner}(ii).
\end{proof}

\begin{prop}\label{prop:outer}We suppose that $H^0$ is not trivial.
Let $\rho$ and $\sigma$ be coactions of $H^0$ on $A$ and $B$, respectively. If
$\rho$ is strongly Morita equivalent to $\sigma$, then
$\rho$ is outer if and only if so is $\sigma$.
\end{prop}
\begin{proof} We suppose that $\rho$ is outer. We show that $\sigma$ is outer.
Let $\pi$ be a surjective $C^*$-Hopf algebra homomorphism of $H^0$ onto
a non-trivial $C^*$-Hopf algebra $K^0$. We suppose that $(\id\otimes\pi)\circ\sigma$ is inner.
Then $(\id\otimes\pi)\circ\sigma$ is strongly Morita equivalent to $(\id\otimes\pi)\circ\rho$
by easy computations. Thus by Corollary \ref{cor:inner}(i), $(\id\otimes\pi)\circ\rho$ is inner.
This is a contradiction. Therefore, we obtain the conclusion.
\end{proof}

Furthermore, we have also the following easy lemma:

\begin{lem}\label{lem:ample}Let $(\rho, u)$ be a twisted coaction of $H^0$ on $A$ and let
$(\rho\otimes\id, u\otimes I_n )$ be a twisted coaction of $H^0$ on $A\otimes M_n (\BC)$, where $n$ is
any positive integer and we identify $A\otimes H^0 \otimes M_n (\BC)$ with $A\otimes M_n (\BC)\otimes H^0$.
Then $(\rho, u)$ is strongly Morita equivalent to $(\rho\otimes\id, u\otimes I_n )$.
\end{lem}
\begin{proof}
Let $f$ be a minimal projection in $M_n (\BC)$ and let $X=(1\otimes f)(A\otimes M_n (\BC))$.
We regard it as an $A-A\otimes M_n (\BC)$-equivalence bimodule in the usual way.
Let $\lambda$ be a linear map from $X$ to $X\otimes H^0$ defined by
$$
\lambda((1\otimes f)x)=(1\otimes f\otimes 1^0 )(\rho\otimes\id)(x)
$$
for any $x\in A\otimes M_n (\BC)$, where we identify $A\otimes H^0 \otimes M_n (\BC)$ with
$A\otimes M_n (\BC)\otimes H^0$. By routine computations, we can see that $\lambda$ satisfies Conditions
(1)-(5) in Definition \ref{Def:wcovariant} and Condition $(**)$.
\end{proof}

\section{Crossed products of Hilbert $C^*$-bimodules of finite type by finite dimensional $C^*$-Hopf algebras}
\label{crossed}In this section, we extend the notion of crossed products of Hilbert $C^*$-bimodules
of finite type defined in \cite {JW:covariant}, \cite {KW2:discrete} to (twisted) coactions of finite dimensional $C^*$-Hopf algebras.
\par
Let $H$ be a finite dimensional $C^*$-Hopf algebra with its dual $C^*$-Hopf algebra $H^0$.
Let $A$ and $B$ be unital $C^*$-algebras and $X$ a Hilbert $A-B$-bimodule of finite
type. Let $(A, B, X, \rho, u, \sigma, v, \lambda, H^0 )$ be a twisted covariant system.
Under certain conditions, we define $X\rtimes_{\lambda}H$, a Hilbert $A\rtimes_{\rho, u}H-B\rtimes_{\sigma, v}H$-bimodule
of finite type as follows: $X\rtimes_{\lambda}H$ is just $X\otimes H$ (the algebraic tensor product)
as vector spaces. Its left action and right action are given by
\begin{align*}
(a\rtimes_{\rho, u}h)(x\rtimes_{\lambda}l)
& =a[h_{(1)}\cdot_{\lambda}x]\widehat{v}(h_{(2)}, l_{(1)})\rtimes_{\lambda}h_{(3)}l_{(2)}, \\
(x\rtimes_{\lambda}l)(b\rtimes_{\sigma, v}m )
& =x[l_{(1)}\cdot_{\sigma, v}b]\widehat{v}(l_{(2)}, m_{(1)})\rtimes_{\lambda}l_{(3)}m_{(2)}
\end{align*}
for any $a\in A$, $b\in B$, $x\in X$ and $h, l, m\in H$. Then
for any $a_1 , a_2 \in A$, $x\in X$, $h,l,m\in H$,
\begin{align*}
& ((a_1 \rtimes_{\rho, u}h)(a_2 \rtimes_{\rho, u}l))(x\rtimes_{\lambda}m) \\
& =a_1 [h_{(1)}\cdot_{\rho, u}a_2 ]\widehat{u}(h_{(2)}, l_{(1)})[h_{(3)}l_{(2)}\cdot_{\lambda}x]
\widehat{v}(h_{(4)}l_{(3)}, m_{(1)})\rtimes_{\lambda}h_{(5)}l_{(4)}m_{(2)} \\
& =a_1 [h_{(1)}\cdot_{\lambda}a_2 [l_{(1)}\cdot_{\lambda}x]]\widehat{v}(h_{(2)}, l_{(2)})
\widehat{v}(h_{(3)}l_{(3)}, m_{(1)})\rtimes_{\lambda}h_{(4)}l_{(4)}m_{(2)} \\
& =a_1 [h_{(1)}\cdot_{\lambda}a_2 [l_{(1)}\cdot_{\lambda}x]\widehat{v}(l_{(2)}, m_{(1)})]\widehat{v}(h_{(2)}, l_{(3)}m_{(2)})
\rtimes_{\lambda}h_{(3)}l_{(4)}m_{(3)} \\
& =(a_1 \rtimes_{\rho, u}h)((a_2 \rtimes_{\rho, u}l)(x\rtimes_{\lambda}m)).
\end{align*}
Also, for any $b_1 , b_2 \in B$, $x\in X$, $h, l, m \in H$,
\begin{align*}
& (x\rtimes_{\lambda}h)((b_1 \rtimes_{\sigma, v}l)(b_2 \rtimes_{\sigma, v}m)) \\
& =x[h_{(1)}\cdot_{\sigma, v}b_1 ][h_{(2)}\cdot_{\sigma, v}[l_{(1)}\cdot_{\sigma, v}b_2 ]]
\widehat{v}(h_{(3)}, l_{(2)})\widehat{v}(h_{(4)}l_{(3)}, m_{(1)}) \\
& \rtimes_{\lambda}h_{(5)}l_{(4)}m_{(2)} \\
& =x[h_{(1)}\cdot_{\sigma, v}b_1 ]\widehat{v}(h_{(2)}, l_{(1)})[h_{(3)}l_{(2)}\cdot_{\sigma, v}b_2 ]
\widehat{v}(h_{(4)}l_{(3)}, m_{(1)})\rtimes_{\lambda}h_{(5)}l_{(4)}m_{(2)} \\
& =((x\rtimes_{\lambda}h)(b_1 \rtimes_{\sigma, v}l))(b_2 \rtimes_{\sigma, v}m) .
\end{align*}
Thus $X\rtimes_{\lambda}H$ is a left $A\rtimes_{\rho, u}H$ and right $B\rtimes_{\sigma, v}H$-bimodule.
Also, its left $A\rtimes_{\rho, u}H$-valued inner product and right $B\rtimes_{\sigma, v}H$-valued inner product
are given by
\begin{align*}
& {}_{A\rtimes_{\rho, u}H}\la x\rtimes_{\lambda}h \, , \, y\rtimes_{\lambda}l \ra \\
& ={}_A \la x \, , \, [S(h_{(2)}l_{(3)}^* )^* \cdot_{\lambda}y]
\widehat{v}(S(h_{(1)}l_{(2)}^* )^* \, , \, l_{(1)}) \ra 
\rtimes_{\rho, u}h_{(3)}l_{(4)}^* , \\
& \la x\rtimes_{\lambda}h \, , \, y\rtimes_{\lambda}l \ra_{B\rtimes_{\sigma, v}H} \\
& =\widehat{v^*}(h_{(2)}^* , S(h_{(1)})^* )[h_{(3)}^* \cdot_{\sigma, v} \la x \, , \, y \ra_B ]
\widehat{v}(h_{(4)}^* , l_{(1)})
\rtimes_{\sigma, v}h_{(5)}^* l_{(2)}
\end{align*}
for any $x, y\in X$ and $h, l\in H$. We shall show that $X\rtimes_{\lambda}H$ is
a Hilbert $A\rtimes_{\rho, u}H-B\rtimes_{\sigma, v}H$-bimodule of finite type proving
that $X\rtimes_{\lambda}H$ satisfies Conditions
(1)-(10) in \cite [Lemma 1.3]{KW2:discrete}. Clearly $X\rtimes_{\lambda}H$ is
a left $A\rtimes_{\rho}H$- and right $B\rtimes_{\sigma}H$-bimodule. Thus Conditions (1), (4) in
\cite [Lemma 1.3]{KW2:discrete} are satisfied. 
For any $a, b\in A$, $x, y\in X$ and $h, l, m\in H$,
\begin{align*}
& (a\rtimes_{\rho, u}h)\,{}_{A\rtimes_{\rho, u}H} \la x\rtimes_{\lambda}l \, , \, y\rtimes_{\lambda}m \ra \\
& =a[h_{(1)}\cdot_{\rho, u} {}_A \la x \, , \, 
[S(l_{(2)}m_{(3)}^* )^* \cdot_{\lambda}y]\widehat{v}(S(l_{(1)}m_{(2)}^* )^* , m_{(1)}) \ra]
\widehat{u}(h_{(2)}, l_{(3)}m_{(4)}^* ) \\
& \rtimes_{\rho, u}h_{(3)}l_{(4)}m_{(5)}^* \\
& =a \,{}_A \la [h_{(1)}\cdot_{\lambda}x] \, , \, [S(h_{(3)}^* \cdot_{\lambda}[S(l_{(2)}m_{(3)}^* )^* \cdot_{\lambda}y]
[S(h_{(2)}^* )\cdot_{\sigma, v}\widehat{v}(S(l_{(1)}m_{(2)}^* )^* , m_{(1)})] \ra \\
& \times\widehat{u}(h_{(4)}, l_{(3)}m_{(4)}^* )\rtimes_{\rho, u}h_{(5)}l_{(4)}m_{(5)}^* \\
& =a\, {}_A \la [h_{(1)}\cdot_{\lambda}x] \, , \, [S(h_{(5)}^* )\cdot_{\lambda}[S(l_{(4)}m_{(4)}^* )^* \cdot_{\lambda}y]]
\widehat{v}(S(h_{(4)}^* ), S(l_{(3)}m_{(3)}^* )^* ) \\
& \times\widehat{v}(S(h_{(3)}l_{(2)}m_{(2)}^* )^* , m_{(1)})\widehat{v^*} (S(h_{(2)}^* ), S(l_{(1)}^* )) \ra
\widehat{u}(h_{(6)}, l_{(5)}m_{(5)}^* ) \\
& \rtimes_{\rho, u}h_{(7)}l_{(6)}m_{(6)}^* \\
& =a\, {}_A \la [h_{(1)}\cdot_{\lambda}x] \, , \, 
\widehat{u}(S(h_{(5)}^* ), S(l_{(4)}m_{(4)}^* )^* )[S(h_{(4)}l_{(3)}m_{(3)}^* )^*
\cdot_{\lambda}y] \\
& \times\widehat{v}(S(h_{(3)}l_{(2)}m_{(2)}^* )^* , m_{(1)})\widehat{v}(h_{(2)}, l_{(1)})^* \ra 
\widehat{u}(h_{(6)}, l_{(5)}m_{(5)}^* )\rtimes_{\rho, u}h_{(7)}l_{(6)}m_{(6)}^* \\
& =a\, {}_A \la [h_{(1)}\cdot_{\lambda}x]\widehat{v}(h_{(2)}, l_{(1)})\, , \, [S(h_{(4)}l_{(3)}m_{(3)}^* )^* \cdot_{\lambda}y]
\widehat{v}(S(h_{(3)}l_{(2)}m_{(2)}^* )^* , m_{(1)}) \ra \\
& \times\widehat{u^* }(h_{(5)}, l_{(4)}m_{(4)}^* )\widehat{u}(h_{(6)}, l_{(5)}m_{(5)}^* )
\rtimes_{\rho, u}h_{(7)}l_{(6)}m_{(6)}^* \\
& =a\,{}_A \la [h_{(1)}\cdot_{\lambda}x]\widehat{v}(h_{(2)}, l_{(1)}) \, , \, [S(h_{(4)}l_{(3)}m_{(3)}^* )^* \cdot_{\lambda}y]
\widehat{v}(S(h_{(3)}l_{(2)}m_{(2)}^* )^* , m_{(1)} ) \ra \\
& \rtimes_{\rho, u}h_{(5)}l_{(4)}m_{(4)}^* \\
& ={}_{A\rtimes_{\rho, u}H} \la a[h_{(1)}\cdot_{\lambda}x]\widehat{v}(h_{(2)}, l_{(1)})\rtimes_{\lambda}h_{(3)}l_{(2)} \, , \,
y\rtimes_{\lambda}m \ra 
\end{align*}
$$
={}_{A\rtimes_{\rho, u}H} \la (a\rtimes_{\rho, u}h)(x\rtimes_{\lambda}l) \, , \, y\rtimes_{\lambda}m \ra .
$$

Also,
\begin{align*}
& \la x\rtimes_{\lambda}h \, , \, y\rtimes_{\lambda}l \ra_{B\rtimes_{\sigma, v}H}(b\rtimes_{\sigma, v}m) \\
& =\widehat{v^*}(h_{(2)}^* , S(h_{(1)}^* ))[h_{(3)}^* \cdot_{\sigma, v} \la x,y \ra_B ]\widehat{v}(h_{(4)}^* , l_{(1)})
[h_{(5)}^* l_{(2)}\cdot_{\sigma, v}b]\widehat{v}(h_{(6)}^* l_{(3)}, m_{(1)}) \\
& \rtimes_{\sigma, v}h_{(7)}^* l_{(4)}m_{(2)} \\
& =\widehat{v^*}(h_{(2)}^* , S(h_{(1)}^* ))[h_{(3)}^* \cdot_{\sigma, v} \la x, y \ra_B ]
[h_{(4)}^* \cdot_{\sigma, v}[l_{(1)}\cdot_{\sigma, v}b]]\widehat{v}(h_{(5)}^* , l_{(2)}) 
\widehat{v}(h_{(6)}^* l_{(3)}, m_{(1)}) \\
& \rtimes_{\sigma, v}h_{(7)}^* l_{(4)}m_{(2)} \\
& =\widehat{v^*}(h_{(2)}^* S(h_{(1)}^* ))[h_{(3)}^* \cdot_{\sigma, v} \la x, y \ra_B [l_{(1)}\cdot_{\sigma, v}b ]] \widehat{v}
(h_{(4)}^* , l_{(2)})\widehat{v}(h_{(5)}^* l_{(3)}, m_{(1)}) \\
& \rtimes_{\sigma, v}h_{(6)}^* l_{(4)}m_{(2)} \\
& =\widehat{v^*}(h_{(2)}^* , S(h_{(1)}^* ))[h_{(3)}^* \cdot_{\sigma, v} \la x, y \ra_B [\l_{(1)}\cdot_{\sigma, v}b]]
[h_{(4)}^* \cdot_{\sigma, v}\widehat{v}(l_{(2)}, m_{(1)})]\widehat{v}(h_{(5)}^* , l_{(3)}m_{(2)}) \\
& \rtimes_{\sigma, v}h_{(6)}^* l_{(4)}m_{(3)} \\
& =\widehat{v^*}(h_{(2)}^* , S(h_{(1)}^* ))[h_{(3)}^* \cdot_{\sigma, v} \la x, y \ra_B [l_{(1)}\cdot_{\sigma, v}b]\widehat{v}
(l_{(2)}, m_{(1)})]\widehat{v}(h_{(4)}^* , l_{(3)}m_{(2)}) \\
& \rtimes_{\sigma, v}h_{(5)}^* l_{(4)}m_{(3)} \\
& =\la x\rtimes_{\lambda}h \, , \, (y\rtimes_{\lambda}l)(b\rtimes_{\sigma, v}m) \ra_{B\rtimes_{\sigma, v}H} .
\end{align*}
Thus Conditions (3), (6) in \cite [Lemma 1.3]{KW2:discrete} are satisfied.
For any $x, y\in X$ and $h, l \in H$,
\begin{align*}
& {}_{A\rtimes_{\rho, u}H} \la x\rtimes_{\lambda}h \, , \, y\rtimes_{\lambda}l \ra^* \\
& =\widehat{u^* }(l_{(5)}h_{(4)}^* , S(l_{(4)}h_{(3)}^* ))
[l_{(6)}h_{(5)}^* \cdot_{\rho, u}{}_A \la [S(l_{(3)}h_{(2)}^* ) \cdot_{\lambda}y]
\widehat{v}(S(l_{(2)}h_{(1)}^* ) , l_{(1)}) \, , \, x \ra ] \\
& \rtimes_{\rho, u}\l_{(7)}h_{(6)}^* \\
& =\widehat{u^*}(l_{(5)}h_{(4)}^* , S(l_{(4)}h_{(3)}^* )) \\
& \times{}_A \la [l_{(6)}h_{(5)}^* \cdot_{\lambda}[S(l_{(3)}h_{(2)}^* )\cdot_{\lambda}y]]
[l_{(7)}h_{(6)}^* \cdot_{\sigma, v}\widehat{v}(S(l_{(2)}h_{(1)}^* ) , l_{(1)})] \, , \, 
[S(l_{(8)}h_{(7)}^* )^* \cdot_{\lambda}x] \ra \\
& \rtimes_{\rho, u}l_{(9)}h_{(8)}^* \\
& =\widehat{u^*}(l_{(6)}h_{(6)}^* , S(l_{(5)}h_{(5)}^* )) \\
& \times {}_A \la [l_{(7)}h_{(7)}^* \cdot_{\lambda}[S(l_{(4)}h_{(4)}^* )\cdot_{\lambda}y]]
\widehat{v}(l_{(8)}h_{(8)}^* , S(l_{(3)}h_{(3)}^* ))\widehat{v}(l_{(9)}h_{(9)}^* S(l_{(2)}h_{(2)}^* ), l_{(1)}) \\
& \times\widehat{v^*}(l_{(10)}h_{(10)}^* , S(h_{(1)}^* )) \, , \, [S(l_{(11)}h_{(11)}^* )^* \cdot_{\lambda}x] \ra
\rtimes_{\rho, u} l_{(12)}h_{(12)}^* \\
& =\widehat{u^*}(l_{(6)}h_{(6)}^* , S(l_{(5)}h_{(5)}^* )) \\
& \times {}_A \la \widehat{u}(l_{(7)}h_{(7)}^* , S(l_{(4)}h_{(4)}^* ))[l_{(8)}h_{(8)}^*S(l_{(3)}h_{(3)}^* )\cdot_{\lambda}y]
\widehat{v}(l_{(9)}h_{(9)}^*S(l_{(2)}h_{(2)}^* ) , l_{(1)}) \\
& \times \widehat{v^*}(l_{(10)}h_{(10)}^* , S(h_{(1)}^* )) \, , \, [S(l_{(11)}h_{(11)}^* )^* \cdot_{\lambda}x] \ra
\rtimes_{\rho, u}l_{(8)}h_{(8)}^* \\
& ={}_A \la y \, , \,[S(l_{(2)}h_{(3)}^* )^* \cdot_{\lambda}x]\widehat{v}(S(l_{(1)}h_{(2)}^* )^* , h_{(1)}) \ra
\rtimes_{\rho, u}l_{(3)}h_{(4)}^* \\
& ={}_{A\rtimes_{\rho, u}H} \la y\rtimes_{\lambda}l \, , \, x\rtimes_{\lambda}h \ra .
\end{align*}
Similarly
\begin{align*}
& \la x\rtimes_{\lambda}h \, , \, y\rtimes_{\lambda}l \ra_{B\rtimes_{\sigma, v}H}^* \\
& =\widehat{v^*}(l_{(3)}^* h_{(6)}, S(l_{(2)}^* h_{(5)}))[l_{(4)}^* h_{(7)}\cdot_{\sigma}\widehat{v}(h_{(4)}^* , l_{(1)})^*
[S(h_{(3)})\cdot_{\sigma, v} \la y, x \ra_B ]\widehat{v}(S(h_{(2)}), h_{(1)})] \\
& \rtimes_{\sigma, v}l_{(5)}^* h_{(8)} \\
& =\widehat{v^*}(l_{(3)}^* h_{(6)}, S(l_{(2)}h_{(5)}))[l_{(4)^* }h_{(7)}\cdot_{\sigma, v}\widehat{v^*}(S(h_{(4)}) , S(l_{(1)}^* ))] \\
& \times [l_{(5)}^* h_{(8)}\cdot_{\sigma, v}[S(h_{(3)})\cdot_{\sigma, v} \la y, x \ra_B ]]
[l_{(6)}^* h_{(9)}\cdot_{\sigma, v}\widehat{v}(S(h_{(2)}), h_{(1)})]\rtimes_{\sigma, v}l_{(7)}^* h_{(10)} \\
& =\widehat{v}(S(l_{(3)}^* h_{(6)})^* , (l_{(2)}^* h_{(5)})^* )^* [S(h_{(7)}^* l_{(4)})\cdot_{\sigma}
\widehat{v}(h_{(4)}^* , l_{(1)})]^* \\
& \times[l_{(5)}^* h_{(8)}\cdot_{\sigma, v}[S(h_{(3)})\cdot_{\sigma, v} \la y, x \ra_B ]]
[l_{(6)}^* h_{(9)}\cdot_{\sigma, v}\widehat{v}(S(h_{(2)}), h_{(1)})]\rtimes_{\sigma, v}l_{(7)}^* h_{(10)} \\
& =[[S(h_{(7)}^* l_{(4)})\cdot_{\sigma, v}\widehat{v}(h_{(4)}^* , l_{(1)})]
\widehat{v}(S(h_{(6)}^* l_{(3)}), h_{(5)}^* l_{(2)})]^*
[l_{(5)}^* h_{(8)}\cdot_{\sigma, v}[S(h_{(3)})\cdot_{\sigma, v} \la y, x \ra_B ]] \\
& \times [l_{(6)}^* h_{(9)}\cdot_{\sigma, v}\widehat{v}(S(h_{(2)}), h_{(1)})]\rtimes_{\sigma, v}l_{(7)}^* h_{(10)} \\
& =[\widehat{v}(S(h_{(7)}^* l_{(3)}), h_{(4)}^* )\widehat{v}(S(h_{(6)}^* l_{(2)})h_{(5)}^* , l_{(1)})]^*
[l_{(4)}^* h_{(8)}\cdot_{\sigma, v}[S(h_{(3)})\cdot_{\sigma, v} \la y, x \ra_B ]] \\
& \times [l_{(5)}^* h_{(9)}\cdot_{\sigma, v}\widehat{v}(S(h_{(2)}), h_{(1)})]\rtimes_{\sigma, v}l_{(6)}^* h_{(10)} \\
& =\widehat{v}(S(l_{(2)}), l_{(1)})^* \widehat{v^*}(l_{(3)}^* h_{(5)}, S(h_{(4)}))[l_{(4)}^* h_{(6)}\cdot_{\sigma, v}
[S(h_{(3)})\cdot_{\sigma, v} \la y, x \ra_B ]] \\
& \times [l_{(5)}^* h_{(7)}\cdot_{\sigma, v}\widehat{v}(S(h_{(2)}), h_{(1)})]\rtimes_{\sigma, v}l_{(6)}^*h_{(8)} \\
& =\widehat{v}(S(l_{(2)}), l_{(1)})^* [l_{(3)}^* h_{(5)}S(h_{(4)})\cdot_{\sigma, v}\la y, x \ra_B ]
\widehat{v^*}(l_{(4)}^* h_{(6)}, S(h_{(3)})) \\
& \times [l_{(5)}^* h_{(7)}\cdot_{\sigma, v}\widehat{v}(S(h_{(2)}), h_{(1)})[\rtimes_{\sigma, v}l_{(6)}^* h_{(8)} \\
& =\widehat{v}(S(l_{(2)}), l_{(1)})^* [l_{(3)}^* \cdot_{\sigma, v} \la y, x \ra_B ]\widehat{v^*}(l_{(4)}^* h_{(4)}, S(h_{(3)}))
[l_{(5)}^* h_{(5)}\cdot_{\sigma, v}\widehat{v}(S(h_{(2)}), h_{(1)})] \\
& \rtimes_{\sigma, v}l_{(6)}^* h_{(6)} \\
& =\widehat{v}(S(l_{(2)}), l_{(1)})^* [l_{(3)}^* \cdot_{\sigma, v} \la y, x \ra_B ]
\widehat{v}(l_{(4)}^* h_{(5)}S(h_{(4)}), h_{(1)})
\widehat{v^*}(l_{(5)}^* h_{(6)}, S(h_{(3)})h_{(2)}) \\
& \rtimes_{\sigma, v}l_{(6)}^* h_{(7)} \\
& =\la y\rtimes_{\lambda}l \, , \, x\rtimes_{\lambda}h \ra_{B\rtimes_{\sigma, v}H} .
\end{align*}
Thus Conditions (2), (5) in \cite [Lemma 1.3]{KW2:discrete} are satisfied.
Moreover, for any $b\in B$, $x, y\in X$, $l, m\in H$,
\begin{align*}
& {}_{A\rtimes_{\rho, u}H} \la x\rtimes_{\lambda}l \, , \, (y\rtimes_{\lambda}m)(b\rtimes_{\sigma, v}1)^* \ra \\
& ={}_A \la x \, , \, [S(l_{(3)}m_{(5)}^* )^* \cdot_{\lambda}y]\widehat{v}(S(l_{(2)}m_{(4)}^* )^* , m_{(1)})
[S(l_{(1)}m_{(3)}^* )^* m_{(2)}\cdot_{\sigma, v}b^* ] \ra\\
& \rtimes_{\rho, u}l_{(4)}m_{(6)}^* \\
& ={}_A \la x[l_{(1)}\cdot_{\sigma, v}b] \, , \, [S(l_{(3)}m_{(3)}^* )^* \cdot_{\lambda}y]
\widehat{v}(S(l_{(2)}m_{(2)}^* )^* , m_{(1)}) \ra\rtimes_{\rho, u}l_{(4)}m_{(4)}^* \\
& ={}_{A\rtimes_{\rho, u}H} \la (x\rtimes_{\lambda}l)(b\rtimes_{\sigma, v}1) \, , \, y\rtimes_{\lambda}m \ra .
\end{align*}
Also, for any $x, y\in X$, $h, l, m\in H$,
\begin{align*}
& {}_{A\rtimes_{\rho, u}H} \la x\rtimes_{\lambda} l \, , \, (y\rtimes_{\lambda}m)(1\rtimes_{\sigma, v}h)^* \ra \\
& ={}_{A\rtimes_{\rho, u}H} \la x\rtimes_{\lambda}l \, , \, y[m_{(1)}\cdot_{\sigma, v}\widehat{v}(S(h_{(2)}), h_{(1)})^* ]
\widehat{v}(m_{(2)}, h_{(3)}^* )\rtimes_{\lambda}m_{(3)}h_{(4)}^* \ra \\
& ={}_A \la x \, , \, [S(l_{(2)}h_{(6)}m_{(5)}^* )^* \cdot_{\lambda}y[m_{(1)}\cdot_{\sigma, v}
\widehat{v^*}(h_{(2)}^* , S(h_{(1)})^* ) ]\widehat{v}(m_{(2)}, h_{(3)}^* )] \\
& \times \widehat{v}(S(l_{(1)}h_{(5)}m_{(4)}^* )^* , m_{(3)}h_{(4)}^* )\rtimes_{\rho, u}l_{(3)}h_{(7)}m_{(6)}^* \\
& ={}_A \la x \, , \, [S(l_{(2)}h_{(7)}m_{(5)}^* )^* \cdot_{\lambda}y\widehat{v}(m_{(1)}, h_{(3)}^* S(h_{(2)}^* ))
\widehat{v^*}(m_{(2)}h_{(4)}^* , S(h_{(1)}^* ))] \\
& \times \widehat{v}(S(l_{(1)}h_{(6)}m_{(4)}^* )^* , m_{(3)}h_{(5)}^* ) \ra \rtimes_{\rho, u}l_{(3)}h_{(8)}m_{(6)}^* \\
& ={}_A \la x \, , \, [S(l_{(3)}h_{(6)}m_{(5)}^* )^* \cdot_{\lambda}y][S(l_{(2)}h_{(5)}m_{(4)}^* )^* \cdot_{\sigma, v}
\widehat{v^* }(m_{(1)}h_{(2)}^* , S(h_{(1)}^* ))] \\
& \times \widehat{v}(S(l_{(1)}h_{(4)}m_{(3)}^* )^* , m_{(2)}h_{(3)}^* ) \ra \rtimes_{\rho, u}l_{(4)}h_{(7)}m_{(6)}^* \\
& ={}_A \la x \, , \, [S(l_{(3)}h_{(7)}m_{(5)}^* )^* \cdot_{\lambda}y]
\widehat{v}(S(l_{(2)}h_{(6)}m_{(4)}^* )^* , m_{(1)}h_{(3)}^* S(h_{(2)}^* )) \\
& \times \widehat{v^*}(S(l_{(1)}h_{(5)}m_{(3)}^* )^* m_{(2)}h_{(4)}^* , S(h_{(1)}^* ))\ra
\rtimes_{\rho, u}l_{(4)}h_{(8)}m_{(6)}^* \\
& ={}_A \la x \widehat{v}(l_{(1)}, h_{(1)}) \, , \, [S(l_{(3)}h_{(3)}m_{(3)}^* )^* \cdot_{\lambda}y]
\widehat{v}(S(l_{(2)}h_{(2)}m_{(2)}^* )^* , m_{(1)}) \ra \\
& \rtimes_{\rho, u}l_{(4)}h_{(4)}m_{(4)}^* \\
& ={}_{A\rtimes_{\rho, u}H} \la (x\rtimes_{\lambda}l)(1\rtimes_{\sigma, v}h) \, , \, y\rtimes_{\lambda}m \ra .
\end{align*}
Thus we obtain that for any $b\in B$, $x, y\in X$, $h, l, m\in H$,
$$
{}_{A\rtimes_{\rho, u}H} \la (x\rtimes_{\lambda} l)(b\rtimes_{\sigma, v}h) \, , \, y\rtimes_{\lambda}m \ra
={}_{A\rtimes_{\rho, u}H} \la x\rtimes_{\lambda} l \, , \, (y\rtimes_{\lambda}m)(b\rtimes_{\sigma, v}h)^* \ra .
$$
We note that for any $a\in A$, $x, y\in X$, $h, l, m \in H$,
\begin{align*}
& \la (a\rtimes_{\rho, u}h)(x\rtimes_{\lambda}l) \, , \, y\rtimes_{\lambda}m \ra_{B\rtimes_{\sigma, v}H} \\
& =(1\rtimes_{\sigma, v}l)^* \la (a\rtimes_{\rho, u}h)(x\rtimes_{\lambda}1) \, , \,
y\rtimes_{\lambda}1 \ra_{B\rtimes_{\sigma, v}H}(1\rtimes_{\sigma, v}m).
\end{align*}
Hence in order to show that for any $a\in A$, $x, y\in X$, $h, l, m\in H$,
$$
\la (a\rtimes_{\rho, u}h)(x\rtimes_{\lambda}l) \, , \, y\rtimes_{\lambda}m \ra_{B\rtimes_{\sigma, v}H}
= \la x\rtimes_{\lambda}l \, , \, (y\rtimes_{\lambda}m) (a\rtimes_{\rho, u}h)^* \ra_{B\rtimes_{\sigma, v}H},
$$
we have only to show that for any $a\in A$, $x, y\in X$, $h\in H$,
$$
\la (a\rtimes_{\rho, u}h)(x\rtimes_{\lambda}1) \, , \, y\rtimes_{\lambda} 1 \ra_{B\rtimes_{\sigma, v}H}
= \la x\rtimes_{\lambda}1 \, , \, (a\rtimes_{\rho, u}h)^* (y\rtimes_{\lambda}1) \ra_{B\rtimes_{\sigma, v}H} .
$$
For any $a\in A$, $x, y\in X$,
\begin{align*}
& \la (a\rtimes_{\rho, u}1)(x\rtimes_{\lambda}1) \, , \, y\rtimes_{\lambda}1 \ra_{B\rtimes_{\sigma, v}H}
=\la ax\rtimes_{\lambda}1 \, , \, y\rtimes_{\lambda}1 \ra_{B\rtimes_{\sigma, v}H} \\
& =\la ax , y \ra_B
=\la x\rtimes_{\lambda}1 \, , \, (a\rtimes_{\rho, u}1)^* (y\rtimes_{\lambda}1) \ra_{B\rtimes_{\sigma, v}H} .
\end{align*}
Also, for any $x, y\in X$, $h\in H$,
\begin{align*}
& \la (1\rtimes_{\rho, u}h)(x\rtimes_{\lambda} 1) \, , \, y\rtimes_{\lambda}1 \ra_{B\rtimes_{\sigma, v}H} \\
& =\la [S(h_{(4)})\cdot_{\lambda}[h_{(1)}\cdot_{\lambda}x]]\widehat{v}(S(h_{(3)}), h_{(2)}) \, , \,
[h_{(5)}^* \cdot_{\lambda}y] \ra_B \rtimes_{\sigma, v}h_{(6)}^* \\
& =\la \widehat{u}(S(h_{(2)}), h_{(1)})x \, , \, [h_{(3)}^* \cdot_{\lambda}y] \ra_B\rtimes_{\sigma, v}h_{(4)}^* \\
& =\la x\rtimes_{\lambda}1 \, , \, \widehat{u}(S(h_{(2)}), h_{(1)})^* [h_{(3)}^* \cdot_{\lambda}y]\rtimes_{\lambda}h_{(4)}^*
\ra_{B\rtimes_{\sigma, v}H} \\
& =\la x\rtimes_{\lambda}1 , \, (1\rtimes_{\rho, u}h)^* (y\rtimes_{\lambda}1 ) \ra_{B\rtimes_{\sigma, v}H} .
\end{align*}
Thus Condition (8) in \cite [Lemma 1.3]{KW2:discrete} is satisfied. Moreover, for any
$a\in A$, $b\in B$, $x\in X$, $h, l, m \in H$,
\begin{align*}
& (a\rtimes_{\rho, u}h)[(x\rtimes_{\lambda}l)(b\rtimes_{\sigma, v}m)] \\
& =a[h_{(1)}\cdot_{\lambda}x][h_{(2)}\cdot_{\sigma, v}[l_{(1)}\cdot_{\sigma, v}b]]\widehat{v}(h_{(3)}, l_{(2)})
\widehat{v}(h_{(4)}l_{(3)}, m_{(1)}) \rtimes_{\lambda}h_{(5)}l_{(4)}m_{(2)} \\
& =a[h_{(1)}\cdot_{\lambda}x]\widehat{v}(h_{(2)}, l_{(1)})[h_{(3)}l_{(2)}\cdot_{\sigma, v}b]
\widehat{v}(h_{(4)}l_{(3)}, m_{(1)})
\rtimes_{\lambda}h_{(5)}l_{(4)}m_{(2)} \\
& =[(a\rtimes_{\rho, u}h)(x\rtimes_{\lambda}l)](b\rtimes_{\sigma, v}m) .
\end{align*}
Thus Condition (7) in \cite [Lemma 1.3]{KW2:discrete} is satisfied. Since $X$ is of finite type,
there are finite subsets $\{w_i \}_{i=1}^n $ and $\{z_j \}_{j=1}^m$ in $X$ such that
$$
x=\sum_{i=1}^n w_i \la w_i , x \ra_B =\sum_{j=1}^m {}_A \la x, z_j \ra z_j
$$
for any $x\in X$. Then we have the following lemma:

\begin{lem}\label{lem:finite}With the above notations, if
$(A, B, X, \rho, \sigma, \lambda, H^0 )$ is a covariant system,
then for any $x\in X$, 
$h\in H$,
\begin{align*}
x\rtimes_{\lambda}h
& =\sum_{i=1}^n (w_i \rtimes_{\lambda}1 )\la w_i \rtimes_{\lambda}1 \, , \, 
x\rtimes_{\lambda}h \ra_{B\rtimes_{\sigma}H} \\
& = \sum_{j=1}^m {}_{A\rtimes_{\rho}H} \la x\rtimes_{\lambda}h \, , \, 
z_j \rtimes_{\lambda}1 \ra (z_j \rtimes_{\lambda} 1).
\end{align*}
\end{lem}
\begin{proof}For any $x\in X$, $h\in H$,
$$
\sum_{i=1}^n (w_i \rtimes_{\lambda}1)\la w_i \rtimes_{\lambda}1 \, , \, x\rtimes_{\lambda}h \ra_{B\rtimes_{\sigma}H}
=\sum_{i=1}^n w_i \la w_i \, , \, x \ra_B \rtimes_{\lambda}h
=x\rtimes_{\lambda}h .
$$
Also,
\begin{align*}
& \sum_{j=1}^m {}_{A\rtimes_{\rho}H} \la x\rtimes_{\lambda}h \, , \, z_j \rtimes_{\lambda}1 \ra (z_j \rtimes_{\lambda}1 ) \\
& =\sum_{j=1}^m{}_A \la [h_{(2)}S(h_{(1)})\cdot_{\lambda}x] \, , \, [S(h_{(3)})^* \cdot_{\lambda}z_j ] \ra
[h_{(4)}\cdot_{\lambda}z_j ]\rtimes_{\lambda}h_{(5)} \\
& =\sum_{j=1}^m [h_{(2)}\cdot_{\lambda}{}_A \la [S(h_{(1)})\cdot_{\lambda}x] \, , \, z_j \ra z_j ]\rtimes_{\lambda}h_{(3)} \\
& =[h_{(2)}\cdot_{\lambda}[S(h_{(1)})\cdot_{\lambda}x]]\rtimes_{\lambda}h_{(3)} =x\rtimes_{\lambda}h .
\end{align*}
Therefore, we obtain the conclusion.
\end{proof}

For any Hilbert $C^*$-bimodule $Y$, $\lInd[Y]$ and $\rInd[Y]$ denote its left index and right index,
respectively.

\begin{cor}\label{cor:index}With the above notations and assumptions,
$$
\lInd[X\rtimes_{\lambda}H]=\lInd[X]\rtimes_{\sigma}1, \quad
\rInd[X\rtimes_{\lambda}H]=\rInd[X]\rtimes_{\rho}1 .
$$
\end{cor}
\begin{proof}
By the definitions of the left index and the right index of a Hilbert $C^*$-bimodule,
\begin{align*}
\lInd[X\rtimes_{\lambda}H] & =\sum_{j=1}^m \la z_j \, , \, z_j \ra_B \rtimes_{\sigma}1 =\lInd[X]\rtimes_{\sigma}1 ,\\
\rInd[X\rtimes_{\lambda}H] & =\sum_{i=1}^n {}_A \la w_i \, , \, w_i \ra\rtimes_{\rho}1 =\rInd[X]\rtimes_{\rho}1 .
\end{align*}
\end{proof}

\begin{prop}\label{prop:type}With the above notations and assupmtions,
$X\rtimes_{\lambda}H$ is a Hilbert
$A\rtimes_{\rho}H-B\rtimes_{\sigma}H$-bimodule of finite type with
$$
\lInd[X\rtimes_{\lambda}H]=\lInd[X]\rtimes_{\sigma}1, \quad \rInd[X\rtimes_{\lambda}H]=\rInd[X]\rtimes_{\rho}1.
$$
\end{prop}
\begin{proof}This is immediate by Lemma \ref{lem:finite}, Corollary \ref{cor:index} and
\cite [Lemma 1.3]{KW2:discrete}.
\end{proof}

\begin{lem}\label{lem:finite2}With the above notations, if
$(A, B, X, \rho, u, \sigma, v, \lambda, H^0 )$ is a twisted covariant system
and $X$ is an $A-B$-equivalence
bimodule, then for any $x\in X$, 
$h\in H$,
\begin{align*}
x\rtimes_{\lambda}h
& =\sum_{i=1}^n (w_i \rtimes_{\lambda}1 )\la w_i \rtimes_{\lambda}1 \, , \, 
x\rtimes_{\lambda}h \ra_{B\rtimes_{\sigma, z}H} \\
& = \sum_{j=1}^m {}_{A\rtimes_{\rho, w}H} \la x\rtimes_{\lambda}h \, , \, 
z_j \rtimes_{\lambda}1 \ra (z_j \rtimes_{\lambda} 1).
\end{align*}
\end{lem}
\begin{proof}For any $x\in X$, $h\in H$,
$$
\sum_{i=1}^n (w_i \rtimes_{\lambda}1)\la w_i \rtimes_{\lambda}1 \, , \, x\rtimes_{\lambda}h \ra_{B\rtimes_{\sigma, v}H}
=\sum_{i=1}^n w_i \la w_i \, , \, x \ra_B \rtimes_{\lambda}h
=x\rtimes_{\lambda}h .
$$
Also,
\begin{align*}
& \sum_{j=1}^m {}_{A\rtimes_{\rho, u}H} \la x\rtimes_{\lambda}h \, , \, z_j \rtimes_{\lambda}1 \ra
(z_j \rtimes_{\lambda}1 ) \\
& =\sum_{j=1}^m {}_A \la x \,  , \, [S(h_{(1)})^* \cdot_{\lambda}z_j ] \ra
[h_{(2)}\cdot_{\lambda}z_j ]\rtimes_{\lambda}h_{(3)} \\
& =\sum_{j=1}^m x \la [S(h_{(1)})^* \cdot_{\lambda}z_j ] \, , \, [h_{(2)}\cdot_{\lambda}z_j ] \ra_B \rtimes_{\lambda}h_{(3)} \\
& =\sum_{j=1}^m x [h_{(1)}\cdot_{\sigma} \la z_j , z_j \ra_B ]\rtimes_{\lambda}h_{(2)}=x\rtimes_{\lambda}h.
\end{align*}
Therefore, we obtain the conclusion.
\end{proof}

\begin{lem}\label{lem:full}With the above notations and assumptions,
if $X$ is an $A-B$-equivalence bimodule,
then the Hilbert $A\rtimes_{\rho, u}H-B\rtimes_{\sigma, v}H$-bimodule is full with the both-sided inner products.
\end{lem}
\begin{proof}
For any $x, y\in X$, ${}_{A\rtimes_{\rho, u}H} \la x\rtimes_{\lambda}1 \, , \, y\rtimes_{\lambda}1 \ra
={}_A \la x\, , \, y \ra \rtimes_{\rho, u}1$. Since
${}_{A\rtimes_{\rho, u} H} \la X\rtimes_{\lambda}H \, , \, X\rtimes_{\lambda}H \ra$
is a closed ideal of $A\rtimes_{\rho, u}H$, for any $x, y\in X$, $h\in H$,
$$
({}_A \la x \, , \, y \ra\rtimes_{\rho, u}1 )(1\rtimes_{\rho, u}h)={}_A \la x \, , \, y\ra\rtimes_{\rho, u}h\in {}_{A\rtimes_{\rho, u}H} \la
X\rtimes_{\lambda} H \, , \, X\rtimes_{\lambda}H \ra .
$$
Since ${}_A \la X \, , \, X \ra=A$, we obtain that
$$
{}_{A\rtimes_{\rho, u}H} \la X\rtimes_{\lambda}H \, , \, X\rtimes_{\lambda}H \ra =A\rtimes_{\rho, u}H .
$$
Also, for any $x, y\in X$, $h\in H$,
$$
\la x\rtimes_{\lambda}1 \, ,\, y\rtimes_{\lambda}h \ra_{B\rtimes_{\sigma, v}H}
=\la x\, , \, y\ra_B \rtimes_{\sigma, v}h \in \la X\rtimes_{\lambda}H \, , \, X\rtimes_{\lambda}H \ra_{B\rtimes_{\sigma, v}H} .
$$
Since $\la X \, , \, X \ra_B =B$, we obtain that
$$
\la X\rtimes_{\lambda}H \, , \, X\rtimes_{\lambda}H \ra_{B\rtimes_{\sigma, v}H}=B\rtimes_{\sigma, v}H .
$$
\end{proof}
  
\begin{cor}\label{cor:bimodule}With the above notations and assumptions, we suppose that
$X$ is an $A-B$-equivalence bimodule. Then
the $X\rtimes_{\lambda}H$ is an $A\rtimes_{\rho, u}H-B\rtimes_{\sigma, v}H$-equivalence bimodule.
\end{cor}
\begin{proof}By Lemma \ref{lem:full}, it suffices to show that
$$
{}_{A\rtimes_{\rho, u}H} \la x\rtimes_{\lambda}h \, , \, y\rtimes_{\lambda}l \ra (z\rtimes_{\lambda}m )
=(x\rtimes_{\lambda}h )\la y\rtimes_{\lambda}l \, , \, z\rtimes_{\lambda}m \ra_{B\rtimes_{\sigma, v}H}
$$
for any $x, y, z\in X$, $h, l, m\in H$. Since $X$ is an $A-B$-equivalence bimodule,
\begin{align*}
& (x\rtimes_{\lambda}h) \la y\rtimes_{\lambda}l \, , \, z\rtimes_{\lambda} m \ra_{B\rtimes_{\sigma, v}H} \\
& =x[h_{(1)}\cdot_{\sigma, v}\widehat{v}^* (l_{(2)}^* , S(l_{(1)})^* )[l_{(3)}^* \cdot_{\sigma, v}
\la y, z \ra_B ]\widehat{v}(l_{(4)}^* , m_{(1)})]\widehat{v}(h_{(2)}, l_{(5)}^* m_{(2)}) \\
& \rtimes_{\lambda}h_{(3)}l_{(6)}^* m_{(3)} \\
& =x[h_{(1)}\cdot_{\sigma, v}\widehat{v}^* (l_{(2)}^* , S(l_{(1)})^* )]
[h_{(2)}\cdot_{\sigma, v}[l_{(3)}^* \cdot_{\sigma, v}\la y, z \ra_B ]] \\
& \times [h_{(3)}\cdot_{\sigma, v}\widehat{v}(l_{(4)}^* , m_{(1)})]
\widehat{v}(h_{(4)}, l_{(5)}^* m_{(2)})\rtimes_{\lambda}h_{(5)}l_{(6)}^* m_{(3)} \\
& =x[h_{(1)}\cdot_{\sigma, v}\widehat{v}^* (l_{(2)}^* , S(l_{(1)}^* ))]\widehat{v}(h_{(2)}, l_{(3)}^* )
[h_{(3)}l_{(4)}^* \cdot_{\sigma, v} \la y, z \ra_B ] \\
& \times\widehat{v}(h_{(4)}l_{(5)}^* , m_{(1)})\rtimes_{\lambda}h_{(5)}l_{(6)}^* m_{(2)} \\
& =x\widehat{v^*}(h_{(1)}l_{(2)}^* , S(l_{(1)}^* ))[h_{(2)}l_{(3)}^* \cdot_{\sigma, v} \la y, z \ra_B ]
\widehat{v}(h_{(3)}l_{(4)}^* , m_{(1)})\rtimes_{\lambda}h_{(4)}l_{(5)}^* m_{(2)}.
\end{align*}
On the other hand,
\begin{align*}
& {}_{A\rtimes_{\rho, u}H} \la x\rtimes_{\lambda}h \, , \, y\rtimes_{\lambda}l \ra (z\rtimes_{\lambda}m ) \\
& ={}_A \la x, [S(h_{(2)}l_{(3)}^* )^* \cdot_{\lambda}y]\widehat{v}(S(h_{(1)}l_{(2)}^* )^* , l_{(1)}) \ra
[h_{(3)}l_{(4)}^* \cdot_{\lambda}z]\widehat{v}(h_{(4)}l_{(5)}^* , m_{(1)}) \\
& \rtimes_{\lambda}h_{(5)}l_{(6)}^* m_{(2)} \\
& =x \la [S(h_{(2)}l_{(3)}^* )^* \cdot_{\lambda}y]\widehat{v}(S(h_{(1)}l_{(2)}^* )^* , l_{(1)}) \, , \, [h_{(3)}l_{(4)}^*
\cdot_{\lambda}z] \ra_B \widehat{v}(h_{(4)}l_{(5)}^* , m_{(1)}) \\
& \rtimes_{\lambda} h_{(5)}l_{(6)}^* m_{(2)} \\
& =x\widehat{v^*}(h_{(1)}l_{(2)}^* , S(l_{(1)}^* ))[h_{(2)}l_{(3)}^* \cdot_{\sigma, v} \la y, z \ra_B ]
\widehat{v}(h_{(3)}l_{(4)}^* , m_{(1)})\rtimes_{\lambda}h_{(4)}l_{(5)}^* m_{(2)}.
\end{align*}
Therefore, we obtain the conclusion.
\end{proof}

By the above discussions, we obtain the following:

\begin{cor}\label{cor:crossed}$(1)$ Let $(A, B, X, \rho, u, \sigma, v, \lambda, H^0 )$
be a twisted covariant system. We suppose that $X$ is an $A-B$-equivalent bimodule. Then the crossed product
$X\rtimes_{\lambda}H$ is an $A\rtimes_{\rho, u}H-B\rtimes_{\sigma, v}H$-equivalence
bimodule.
\newline
$(2)$ Let $(A, B, X, \rho, \sigma, \lambda, H^0 )$ be a covariant system. Then the crossed
$X\rtimes_{\lambda}H$ is a Hilbert $A\rtimes_{\rho}H-B\rtimes_{\sigma}H$-bimodule
of finite type.
\end{cor}

In the situation of Corollary \ref{cor:crossed}(1), let $X\rtimes_{\lambda}H$ be the crossed product
associated to a twisted covariant system $(A, B, X, \rho, u, \sigma, v, \lambda, H^0)$,
where $X$ is an $A-B$-equivalence bimodule. Then
we define the dual covariant system with $X\rtimes_{\lambda}H$ as follows:
Let $\widehat{\rho}$ and $\widehat{\sigma}$ be the dual coactions of $H$ on $A\rtimes_{\rho, u}H$
and $B\rtimes_{\sigma, v}H$ of $(\rho, u)$ and $(\sigma, v)$, respectively. Let $\widehat{\lambda}$ be the
dual coaction of $H$ on $X\rtimes_{\lambda}H$ defined by
$$
\widehat{\lambda}(x\rtimes_{\lambda}h)=(x\rtimes_{\lambda}h_{(1)})\otimes h_{(2)}
$$
for any $x\in X$, $h\in H$. Then by easy computations, we can see that
$$
(A\rtimes_{\rho, u}H\, , \, B\rtimes_{\sigma, v}H\, , \, X\rtimes_{\lambda}H \, , \, 
\widehat{\rho}\, , \, \widehat{\sigma} \, , \, \widehat{\lambda} \, , \, H)
$$
is a covariant system. Hence we obtain the following:

\begin{cor}\label{cor:dual}Let $(\rho, u)$ and $(\sigma, v)$ be twisted coactions of $H^0$ on $A$ and $B$, respectively.
Then the following conditions are equivalent:
\newline
$(1)$ The twisted coaction $(\rho, u)$ is strongly Morita equivalent to the twisted coaction $(\sigma, v)$,
\newline
$(2)$ The dual coaction $\widehat{\rho}$ of $(\rho, u)$ is strongly Morita equivalent to
the dual coaction $\widehat{\sigma}$ of $(\sigma, v)$.
\end{cor}
\begin{proof}By the above discussion, it is clear that Condition (1) implies Condition (2). We suppose Condition (2).
Then by Condition (2), we can see that $\widehat{\widehat{\rho}}$ is strongly Morita equivalent to
$\widehat{\widehat{\sigma}}$, where $\widehat{\widehat{\rho}}$ and $\widehat{\widehat{\sigma}}$
are the dual coactions of $\widehat{\rho}$ and $\widehat{\sigma}$, respectively.
By \cite [Theorem 3.3]{KT2:coaction}, there is an isomorphism $\Psi$
of $M_N (A)$ onto $A\rtimes_{\rho, u}H\rtimes_{\widehat{\rho}}H^0$ such that
$\widehat{\widehat{\rho}}$
is exterior equivalent to the twisted coaction
$$
((\Psi\otimes\id)\circ(\rho\otimes\id)\circ\Psi^{-1}\, , \, (\Psi\otimes\id_{H^0}\otimes\id_{H^0})(u\otimes I_N )).
$$
Hence by Lemma \ref{lem:exterior},
$\widehat{\widehat{\rho}}$ is strongly Morita equivalent to $(\rho\otimes\id, u\otimes I_N)$.
Thus, by Lemma \ref{lem:ample}
and Corollary \ref{cor:relation}, $\widehat{\widehat{\rho}}$ is strongly Morita equivalent to $(\rho, u)$.
Similarly $\widehat{\widehat{\sigma}}$ is strongly Morita equivalent to $(\sigma, v)$. Therefore,
by Corollary \ref{cor:relation}, $(\rho, u)$ is strongly Morita equivalent to $(\sigma, v)$.
\end{proof}

Also, in the situation of Corollary \ref{cor:crossed}(2), we can see that
$$
(A\rtimes_{\rho}H, \, B\rtimes_{\sigma}H, \, X\rtimes_{\lambda}H, \, \widehat{\rho}, \,
\widehat{\sigma}, \, \widehat{\lambda}, H)
$$
is a covariant system in the same way as above.

\section{Duality}\label{sec:duality}
In this section, we present a duality theorem for a crossed product of a Hilbert $C^*$-bimodule
of finite type by a (twisted) coaction of a finite dimensional $C^*$-Hopf algebra in the same way as in
\cite {KT2:coaction}. As mentioned in Section \ref{sec:intro}, Guo and Zhang have
already obtained a duality result using the language of multiplicative
unitary elements and Kac systems in \cite {GZ:Kac}. But we give our duality result because
our approach to coactions of a finite dimensional $C^*$-Hopf algebra on a unital
$C^*$-algebra is a useful addition to the main result in Section \ref {sec:rohlin}.
First, we suppose Condition (1) or Condition (2) in Corollary \ref{cor:crossed}.
In the both cases, we can consider the  dual covariant systems
$$
(A\rtimes_{\rho, u}H\, , \, B\rtimes_{\sigma, v}H\, , \, X\rtimes_{\lambda}H \, , \, 
\widehat{\rho}\, , \, \widehat{\sigma} \, , \, \widehat{\lambda} \, , \, H),
$$
$$
(A\rtimes_{\rho}H, \, B\rtimes_{\sigma}H, \, X\rtimes_{\lambda}H, \, \widehat{\rho}, \,
\widehat{\sigma}, \, \widehat{\lambda}, H).
$$
Let $\Lambda$ be the set
of all triplets $(i, j, k)$ where $i,j=1,2,\dots,d_k$ and $k=1,2,\dots,K$ and $\sum_{k=1}^Kd_k^2 =N$.
For each $I=(i, j, k)\in \Lambda$, let $W_I^{\rho}$, $V_I^{\rho}$ be elements in
$A\rtimes_{\rho, u}H\rtimes_{\widehat{\rho}}H^0$ defined by
$$
W_I^{\rho}=\sqrt{d_k}\rtimes_{\rho, u}w_{ij}^k , \quad
V_I^{\rho}=(1\rtimes_{\rho, u}1\rtimes_{\widehat{\rho}}\tau)(W_I^{\rho}\rtimes_{\widehat{\rho}}1^0 ) .
$$
Similarly for each $I=(i, j, k)\in\Lambda$, we define elements
$$
W_I^{\sigma}=\sqrt{d_k}\rtimes_{\sigma, v}w_{ij}^k , \quad
V_I^{\sigma}=(1\rtimes_{\sigma, v}1\rtimes_{\widehat{\sigma}}\tau)(W_I^{\sigma}\rtimes_{\widehat{\sigma}}1^0 )
$$
in $B\rtimes_{\sigma, v}H\rtimes_{\widehat{\sigma}}H^0$. We regard $M_N (\BC)$ as a Hilbert
$M_N (\BC)-M_N (\BC)$-bimodule in the usual way. Let $X\otimes M_N (\BC)$ be an exterior
tensor product of $X$ and $M_N (\BC)$, which is a Hilbert $A\otimes M_N (\BC)-B\otimes M_N (\BC)$-bimodule.
In the same way as Lemma \ref{lem:tensor}, we can see that $X\otimes M_N (\BC)$ is of finite type.
Let $\{f_{IJ}\}_{I,J\in\Lambda}$ be a system of matrix units of $M_N (\BC)$. Let $\Psi_X$ be a linear map
from $X\otimes M_N (\BC)$ to $X\rtimes_{\lambda}H\rtimes_{\widehat{\lambda}}H^0$ defined by
$$
\Psi_X (\sum_{I, J}x_{IJ}\otimes f_{IJ})
=\sum_{I, J}V_I^{\rho*}(x_{IJ}\rtimes_{\lambda}1\rtimes_{\widehat{\lambda}}1^0 )V_J^{\sigma}
$$
for any $x_{IJ}\in X$. Let $\Psi_A$ and $\Psi_B$ be isomorphisms of $A\otimes M_N(\BC)$ and $B\otimes M_N (\BC)$
onto $A\rtimes_{\rho, u}H\rtimes_{\widehat{\rho}}H^0$ and $B\rtimes_{\sigma, v}H\rtimes_{\widehat{\sigma}}H^0$
defined by
\begin{align*}
\Psi_A (\sum_{I, J}a_{IJ}\otimes f_{IJ})
& =\sum_{I, J}V_I^{\rho*}(a_{IJ}\rtimes_{\rho, u}1\times_{\widehat{\rho}}1^0 )V_J^{\rho}, \\
\Psi_B (\sum_{I, J}b_{IJ}\otimes f_{IJ})
& =\sum_{I, J}V_I^{\sigma*}(a_{IJ}\rtimes_{\sigma, v}1\times_{\widehat{\sigma}}1^0 )V_J^{\sigma}
\end{align*}
for any $a_{IJ}\in A$, $b_{IJ}\in B$, respectively (see \cite{KT2:coaction}).

\begin{lem}\label{lem:module}With the above notations,
\begin{align*}
(1)\, & \Psi_X ((\sum_{I, J}a_{IJ}\otimes f_{IJ})(\sum_{I, J}x_{IJ}\otimes f_{IJ})) \\
& =\Psi_A (\sum_{I, J}a_{IJ}\otimes f_{IJ})\Psi_X (\sum_{I, J}x_{IJ}\otimes f_{IJ}) , \\
(2)\, & \Psi_X ((\sum_{I, J}x_{IJ}\otimes f_{IJ})(\sum_{I, J}b_{IJ}\otimes f_{IJ})) \\
& =\Psi_X (\sum_{I, J}x_{IJ}\otimes f_{IJ})\Psi_B (\sum_{I, J}b_{IJ}\otimes f_{IJ}) , \\
(3)\, & {}_{A\rtimes_{\rho, u}H\rtimes_{\widehat{\rho}}H^0}
\la \Psi_X (\sum_{I, J}x_{IJ}\otimes f_{IJ}) \, , \, \Psi_X (\sum_{I, J}  y_{IJ}\otimes f_{IJ}) \ra \\
& = \Psi_A ({}_{A\otimes M_N (\BC)} \la \sum_{I, J}x_{IJ}\otimes f_{IJ}\, , \, \sum_{I, J}y_{IJ}\otimes f_{IJ} \ra ) , \\
(4)\, & \la \Psi_X (\sum_{I, J}x_{IJ}\otimes f_{IJ}) \, , \,  \Psi_X(\sum_{I, J} y_{IJ}\otimes f_{IJ})
\ra_{B\rtimes_{\sigma, v}H\rtimes_{\widehat{\sigma}}H^0} \\
& = \Psi_B ( \la \sum_{I, J}x_{IJ}\otimes f_{IJ}\, , \,
\sum_{I, J}y_{IJ}\otimes f_{IJ} \ra_{B\otimes M_N (\BC)} )
\end{align*}
for any $a_{IJ}\in A$, $b_{IJ}\in B$, $x_{IJ}\, , \, y_{IJ}\in X$, $I, J\in \Lambda$.
\end{lem}
\begin{proof}
This is immediate by routine computations. Indeed,
$$
\Psi_X ((\sum_{I, J}a_{IJ}\otimes f_{IJ})(\sum_{I , J}x_{I J}\otimes f_{I J}))
=\sum_{I, J, L}V_I^{\rho*}(a_{IL}x_{LJ}\rtimes_{\lambda}1\rtimes_{\widehat{\lambda}}1^0 )V_J^{\sigma} .
$$
On the other hand, by \cite [Lemma 3.1]{KT2:coaction}
\begin{align*}
& \Psi_A (\sum_{I, J}a_{IJ}\otimes f_{IJ})\Psi_X (\sum_{L, M}x_{LM}\otimes f_{LM}) \\
& =\sum_{I, J, M}V_I^{\rho*}(a_{IJ}\rtimes_{\rho, u}1\rtimes_{\widehat{\rho}}1^0 )
(1\rtimes_{\rho, u}1\rtimes_{\widehat{\rho}}\tau)
(x_{JM}\rtimes_{\lambda}1\rtimes_{\widehat{\lambda}}1^0 )V_M^{\sigma} \\
& =\sum_{I, J, M}V_I^{\rho*}(a_{IJ}x_{JM}\rtimes_{\lambda}1\rtimes_{\widehat{\lambda}}1^0 )V_M^{\sigma} .
\end{align*}
Thus we obtain Equation (1). Similarly we can obtain the Equation (2).
Also, by \cite [Lemma 3.1]{KT2:coaction}
\begin{align*}
& {}_{A\rtimes_{\rho, u}H\rtimes_{\widehat{\rho}}H^0} \la \Psi_X (\sum_{I, J}x_{IJ}\otimes f_{IJ}) \, , \,
\Psi_X (\sum_{I, J}  y_{IJ}\otimes f_{IJ}) \ra \\
& =\sum_{I, J, I_1 , J_1}{}_{A\rtimes_{\rho, u}H\rtimes_{\widehat{\rho}}H^0}
\la V_I^{\rho*}(x_{IJ}\rtimes_{\lambda}1\rtimes_{\widehat{\lambda}}1^0)V_J^{\sigma} \, , \, 
V_{I_1}^{\rho*}(y_{I_1 J_1}\rtimes_{\lambda}1\rtimes_{\widehat{\lambda}}1^0 )V_{J_1}^{\sigma} \ra \\
& =\sum_{I, J, I_1}V_I^{\rho*}({}_{A\rtimes_{\rho, u}H} \la x_{IJ}\rtimes_{\lambda}1 \, , \, y_{I_1 J}\rtimes_{\lambda}1 \ra
\rtimes_{\widehat{\rho}}\tau)V_{I_1}^{\rho} \\
& =\sum_{I, J, I_1}V_I^{\rho*}({}_A \la x_{IJ} \, , \, y_{I_1 J} \ra\rtimes_{\rho, u}
\rtimes_{\widehat{\rho}}1^0 )V_{I_1}^{\rho} .
\end{align*}
On the other hand,
\begin{align*}
& \Psi_A (\, {}_{A\otimes M_N (\BC)} \la \sum_{I, J}x_{IJ}\otimes f_{IJ} \, , \, \sum_{I_1 , J_1}y_{I_1 , J_1}
\otimes f_{I_1 , J_1} \ra ) \\
& =\sum_{I, J, I_1}\Psi_A ({}_A \la x_{IJ} \, , \, y_{I_1 J} \ra\otimes f_{II_1}) \\
& =\sum_{I, J, I_1}V_I^{\rho*}({}_A \la x_{IJ} \, , \, y_{I_1 J} \ra\rtimes_{\rho, u}1\rtimes_{\widehat{\rho}}1^0 )
V_{I_1}^{\rho} .
\end{align*}
Thus we obtain Equation (3). Furthermore,
\begin{align*}
& \la \Psi_X (\sum_{I, J}x_{IJ}\otimes f_{IJ}) \, , \, \Psi_X (\sum_{I_1 , J_1}y_{I_1 J_1}\otimes f_{I_1 J_1})
\ra_{B\rtimes_{\sigma, v}H\rtimes_{\widehat{\sigma}}H^0} \\
& =\sum_{I, J, J_1}V_J^{\sigma*}(\la x_{IJ} \, , \, y_{IJ_1} \ra_B\rtimes_{\sigma, v}1\rtimes_{\widehat{\sigma}}1^0 )
V_{J_1}^{\sigma} \\
& =\Psi_B (\la \sum_{I, J}x_{IJ}\otimes f_{IJ} \, , \, \sum_{I_1 , J_1}y_{I_1 J_1}\otimes f_{I_i J_i}
\ra_{B\otimes M_N (\BC)}) .
\end{align*}
Thus we obtain Equation (4).
\end{proof}
By the above lemma, we can see that $\Psi_X$ is injective. Next, we show that $\Psi_X$ is surjective.

\begin{lem}\label{lem:onto}With the above notations, 
$$
(X\rtimes_{\lambda}H\rtimes_{\widehat{\lambda}}1^0 )(1\rtimes_{\sigma, v}1\rtimes_{\widehat{\sigma}}\tau)
(B\rtimes_{\sigma, v}H\rtimes_{\widehat{\sigma}}1^0 )=X\rtimes_{\lambda}H\rtimes_{\widehat{\lambda}}H^0 .
$$
\end{lem}
\begin{proof}
Let $x\in X$, $h\in H$, $\phi\in H^0$. Since
$$
\sum_{i,j,k}(\sqrt{d_k}\rtimes_{\sigma, v}w_{ij}^k \rtimes_{\widehat{\sigma}}1^0 )^* 
(1\rtimes_{\sigma, v}1\rtimes_{\widehat{\sigma}}\tau)(\sqrt{d_k}\rtimes_{\sigma, v}w_{ij}^k \rtimes_{\widehat{\sigma}}1^0 )
=1\rtimes_{\sigma, v}1\rtimes_{\widehat{\sigma}}1^0
$$
by \cite [Proposition 3.18]{KT1:inclusion},
\begin{align*}
& x\rtimes_{\lambda}h\rtimes_{\widehat{\lambda}}\phi \\
& =\sum_{i,j,k}(x\rtimes_{\lambda}h\rtimes_{\widehat{\lambda}}\phi)
(\sqrt{d_k}\rtimes_{\sigma, v}w_{ij}^k \rtimes_{\widehat{\sigma}}1^0 )^* 
(1\rtimes_{\sigma, v}1\rtimes_{\widehat{\sigma}}\tau) \\
& \times (\sqrt{d_k}\rtimes_{\sigma, v}w_{ij}^k 
\rtimes_{\widehat{\sigma}}1^0 ) \\
& =\sum_{i, j, k, j_1 , j_2}d_k ((x\rtimes_{\lambda}h)[\phi\cdot_{\widehat{\sigma}}
(\widehat{v}(S(w_{j_1 j_2}^k ), w_{ij_1}^k )^*
\rtimes_{\sigma, v}w_{j_2 j}^{k*})]\rtimes_{\widehat{\lambda}}\tau) \\
& \times (1\rtimes_{\sigma, v}w_{ij}^k \rtimes_{\widehat{\sigma}}1^0 ) \\
& =\sum_{i, j, k, j_1 , j_2 , j_3}d_k \phi(w_{j_3 j}^{k*})((x\rtimes_{\lambda}h)(\widehat{v}(S(w_{j_1 j_2}^k ), w_{ij_1}^k )^*
\rtimes_{\sigma, v}w_{j_2 j_3}^{k*})\rtimes_{\widehat{\lambda}}1^0 ) \\
& \times (1\rtimes_{\sigma, v}1\rtimes_{\widehat{\sigma}}\tau)
(1\rtimes_{\sigma, v}w_{ij}^k \rtimes_{\widehat{\sigma}}1^0 ).
\end{align*}
Therefore we obtain the conclusion.
\end{proof}

Let $E_1^{\sigma}$ be the canonical conditional expectation from $B\rtimes_{\sigma, v}H$
onto $B$ defined by $E_1^{\sigma}(b\rtimes_{\sigma, v}h)=\tau(h)b$ for any $b\in B$, $h\in H$.
Let $E_1^{\lambda}$ be a linear map from $X\rtimes_{\lambda}H$ onto $X$ defined by
$$
E_1^{\lambda}(x\rtimes_{\lambda}h)=\tau(h)x
$$
for any $x\in X$, $h\in H$. 

\begin{lem}\label{lem:inverse}With the above notations, for any $x\in X$, $h\in H$,
$$
\sum_{i, j, k} (\sqrt{d_k }\rtimes_{\rho, u}w_{ij}^k )^* E_1^{\lambda}((\sqrt{d_k}\rtimes_{\rho, u}w_{ij}^k )
(x\rtimes_{\lambda}h))=x\rtimes_{\lambda}h .
$$
\end{lem}
\begin{proof}
This is also immediate by routine computations. Indeed, for any $x\in X$, $h\in H$,
by \cite [Theorem 2.2]{SP:saturated},
\begin{align*}
& \sum_{i, j, k}(\sqrt{d_k}\rtimes_{\rho, u}w_{ij}^k )^* E_1^{\lambda}((\sqrt{d_k}\rtimes_{\rho, u}w_{ij}^k )
(x\rtimes_{\lambda}h)) \\
& =\sum_{i, j, k, j_1 , j_2 , s, s_1 , s_2 , s_3}d_k \widehat{u^*}(w_{ss_1}^{k*}, w_{si}^k )[w_{s_1 s_2}^{k*}\cdot_{\lambda}
[w_{ij_1}^k \cdot_{\lambda}x]][w_{s_2 s_3}^{k*}\cdot_{\sigma, v}\widehat{v}(w_{j_1 j_2}^k , h_{(1)})] \\
& \rtimes_{\lambda}\tau(w_{j_2 j}^k h_{(2)})w_{s_3 j}^{k*} \\
& =\sum_{i, j, k, j_1 , j_2 , s_2, s_3}d_k x\widehat{v^*}(w_{is_2}^{k*}, w_{ij_1}^k )[w_{s_2 s_3}^{k*}\cdot_{\sigma, v}
\widehat{v}(w_{j_1 j_2}^k , h_{(1)})]\rtimes_{\lambda}\tau(w_{j_2 j}^k h_{(2)})w_{s_3 j}^{k*} \\
& =\sum_{i, j, k, j_1 , j_2 , s_2, s_3}d_k x\widehat{v}(w_{is_2}^{k*}w_{ij_1}^k , h_{(1)})
\widehat{v^*}(w_{s_2 s_3}^{k*}, w_{j_1 j_2}^k h_{(2)})\tau(w_{j_2 j}^k h_{(3)})\rtimes_{\lambda}w_{s_3 j}^{k*} \\
& =\sum_{j, k, s_2}d_k x\tau(w_{s_2 j}^k h)\rtimes_{\lambda}S(w_{js_2}^k ) 
=\sum_{i, k, s_2}d_k x\tau(w_{s_2 j}^k h_{(1)})\rtimes_{\lambda}S(w_{j s_2}^k h_{(2)}S(h_{(3)})) \\
& =x\rtimes_{\lambda}S(\tau(Neh_{(1)})S(h_{(2)}))=x\rtimes_{\lambda}h .
\end{align*}
Therefore, we obtain the conclusion.
\end{proof}

\begin{lem}\label{lem:commute}With the above notations,
$$
(1\rtimes_{\rho, u}1\rtimes_{\widehat{\rho}}\phi)(x\rtimes_{\lambda}1\rtimes_{\widehat{\lambda}}1^0 )
=x\rtimes_{\lambda}1\rtimes_{\widehat{\lambda}}\phi=(x\rtimes_{\lambda}1\rtimes_{\widehat{\lambda}}1^0 )
(1\rtimes_{\sigma, v}1\rtimes_{\widehat{\sigma}}\phi)
$$
for any $x\in X$, $\phi\in H^0$.
\end{lem}
\begin{proof}
For any $x\in X$, $\phi\in H^0$,
\begin{align*}
(1\rtimes_{\rho, u}1\rtimes_{\widehat{\rho}}\phi)(x\rtimes_{\lambda}1\rtimes_{\widehat{\lambda}}1^0 )
& =[\phi_{(1)}\cdot_{\widehat{\lambda}}(x\rtimes_{\lambda}1)]\rtimes_{\widehat{\lambda}}\phi_{(2)}
=x\rtimes_{\lambda}1\rtimes_{\widehat{\lambda}}\phi \\
& =(x\rtimes_{\lambda}1\rtimes_{\widehat{\lambda}}1^0 )
(1\rtimes_{\sigma, v}1\rtimes_{\widehat{\sigma}}\phi).
\end{align*}
\end{proof}

\begin{lem}\label{lem:onto2}With the above notations, $\Psi_X$ is surjective.
\end{lem}
\begin{proof}By Lemma \ref{lem:onto}, it suffices to show that for any $b\in B$, $x\in X$, $h, l\in H$, there is an
element $y\in X\otimes M_N (\BC)$ such that
$$
\Psi_X (y)=(x\rtimes_{\lambda}h\rtimes_{\widehat{\lambda}}1^0 )(1\rtimes_{\sigma, v}1\rtimes_{\widehat{\sigma}}\tau)
(b\rtimes_{\sigma, v}l\rtimes_{\widehat{\sigma}}1^0 ) .
$$
By Lemma \ref{lem:inverse} and \cite [Proposition 3.18]{KT1:inclusion}
\begin{align*}
x\rtimes_{\lambda}h & = \sum_I W_I^{\rho*}(E_1^{\lambda}(W_I^{\rho}(x\rtimes_{\lambda}h))\rtimes_{\lambda}1), \\
b\rtimes_{\sigma, v}l & =\sum_I (E_1^{\sigma}((b\rtimes_{\sigma, v}l)W_I^{\sigma*})\rtimes_{\sigma, v}1 )W_I^{\sigma} .
\end{align*}
Thus
\begin{align*}
& (x\rtimes_{\lambda}h\rtimes_{\widehat{\lambda}}1^0 )(1\rtimes_{\sigma, v}1\rtimes_{\widehat{\sigma}}\tau)
(b\rtimes_{\sigma, v}l\rtimes_{\widehat{\sigma}}1^0 ) \\
& =\sum_{I, J}(W_I^{\rho*}\rtimes_{\widehat{\rho}}1^0 )
(E_1^{\lambda}(W_I^{\rho}(x\rtimes_{\lambda}h))\rtimes_{\lambda}1
\rtimes_{\widehat{\lambda}}\tau) \\
& \times (E_1^{\sigma}((b\rtimes_{\sigma, v}l)W_J^{\sigma*})
\rtimes_{\sigma, v}1\rtimes_{\widehat{\sigma}}\tau)
(W_J^{\sigma}\rtimes_{\widehat{\sigma}}1^0 ) .
\end{align*}
Since
$$
E_1^{\lambda}(W_I^{\rho} (x\rtimes_{\lambda}h))\rtimes_{\lambda}1\rtimes_{\widehat{\lambda}}\tau
=(1\rtimes_{\rho, u}1\rtimes_{\widehat{\rho}}\tau)(E_1^{\lambda}(W_I^{\rho} (x\rtimes_{\lambda}h))
\rtimes_{\lambda}1\rtimes_{\widehat{\lambda}}1^0 )
$$
by Lemma \ref{lem:commute},
\begin{align*}
& (x\rtimes_{\lambda}h\rtimes_{\widehat{\lambda}}1^0 )(1\rtimes_{\sigma, v}1\rtimes_{\widehat{\sigma}}\tau)
(b\rtimes_{\sigma, v}l\rtimes_{\widehat{\sigma}}1^0 ) \\
& =\sum_{I, J}V_I^{\rho*}[E_1^{\lambda}(W_I^{\rho}(x\rtimes_{\lambda}h))
E_1^{\sigma}((b\rtimes_{\sigma, v}l)W_J^{\sigma*})
\rtimes_{\lambda}1\rtimes_{\widehat{\lambda}}1^0 ]V_J^{\sigma} .
\end{align*}
Since $E_1^{\lambda}(W_I^{\rho}(x\rtimes_{\lambda}h))E_1^{\sigma}((b\rtimes_{\sigma, v}l)W_J^{\sigma*})\in X$, we obtain the conclusion.
\end{proof}

Let $\widehat{V^{\rho}}$ be a linear map from $H$ to $A\rtimes_{\rho, u}H$ defined by
$\widehat{V^{\rho}}(h)=1\rtimes_{\rho, u}h$ for any $h\in H$. By \cite{KT1:inclusion},
$\widehat{V^{\rho}}$ is a unitary element in $\Hom (H, A\rtimes_{\rho, u}H)$. Let $V^{\rho}$ be
the unitary element in $(A\rtimes_{\rho, u}H)\otimes H^0$ induced by $\widehat{V^{\rho}}$.
Similarly, we also define unitary elements $\widehat{V^{\sigma}}\in \Hom (H, B\rtimes_{\sigma, v}H)$
and $V^{\sigma}\in (B\rtimes_{\sigma, v}H)\otimes H^0$.

\begin{lem}\label{lem:action}With the above notations, for any $x\in X$, $h\in H$,
$$
[h\cdot_{\lambda}x]\rtimes_{\lambda}1 =\widehat{V^{\rho}}(h_{(1)})(x\rtimes_{\lambda}1)\widehat{V^{\sigma*}}(h_{(2)}).
$$
\end{lem}
\begin{proof}This is also immediate by routine computations. Indeed, for any $x\in X$, $h\in H$,
\begin{align*}
& \widehat{V^{\rho}}(h_{(1)})(x\rtimes_{\lambda}1)\widehat{V^{\sigma*}}(h_{(2)}) \\
& =[h_{(1)}\cdot_{\lambda}x][h_{(2)}\cdot_{\sigma, v}\widehat{v^*}(S(h_{(7)}), h_{(8)})]\widehat{v}(h_{(3)}, S(h_{(6)}))
\rtimes_{\lambda}h_{(4)}S(h_{(5)}) \\
& =[h_{(1)}\cdot_{\lambda}x]\widehat{v}(h_{(2)}, S(h_{(5)})h_{(6)})\widehat{v^*}(h_{(3)}S(h_{(4)}), h_{(7)})
\rtimes_{\lambda}1
=[h\cdot_{\lambda}x]\rtimes_{\lambda}1.
\end{align*}
\end{proof}

\begin{thm}\label{thm:duality}{\rm (Cf. Guo and Zhang \cite [Theorem 2.7]{GZ:Kac})}
Let $A, B$ be unital $C^*$-algebras and $H$ a finite dimensional $C^*$-Hopf algebra
with its dual $C^*$-Hopf algebra $H^0$. Then the following hold:
\newline
$(1)$ If $X$ is an $A-B$-equivalence bimodule and
$( A, B, X, \rho, u, \sigma, v, \lambda$, $H^0 )$ is a twisted covariant system,
then there is a linear isomorphism $\Psi_X$ from $X\otimes M_N (\BC)$
onto $X\rtimes_{\lambda}H\rtimes_{\widehat{\lambda}}H^0$ which satisfies Conditions (1)-(4) in
Lemma \ref{lem:module}, where $X\rtimes_{\lambda}H\rtimes_{\widehat{\lambda}}H^0$ is
an $A\rtimes_{\rho, u}H\rtimes_{\widehat{\rho}}H^0 -B\rtimes_{\sigma, v}H\rtimes_{\widehat{\sigma}}H^0$
-equivalence bimodule and $X\otimes M_N (\BC)$ is an exterior tensor product of an $A-B$-equivalence bimodule $X$
and an $M_N (\BC)-M_N (\BC)$-equivalence bimodule $M_N (\BC)$.
Furthermore, there are unitary elements
$U\in (A\rtimes_{\rho, u}H\rtimes_{\widehat{\rho}}H^0 )\otimes H^0$ and
$V\in (B\rtimes_{\sigma, v}H\rtimes_{\widehat{\sigma}}H^0 )\otimes H^0$
such that
$$
U\widehat{\widehat{\lambda}}(x)V^* =((\Psi_X \otimes\id)\circ(\lambda\otimes\id_{M_N (\BC)})\circ\Psi_X^{-1})(x)
$$
for any $x\in X\rtimes_{\lambda}H\rtimes_{\widehat{\lambda}}H^0$.
\newline
$(2)$ If $X$ is a Hilbert $A-B$-bimodule of finite type and
$( A, B, X,  \rho, \sigma, \lambda$, $H^0 )$ is a covariant system, then
there is a linear isomorphism $\Psi_X$ from $X\otimes M_N (\BC)$
onto $X\rtimes_{\lambda}H\rtimes_{\widehat{\lambda}}H^0$ which satisfies Conditions (1)-(4) in
Lemma \ref{lem:module},
where $X\rtimes_{\lambda}H\rtimes_{\widehat{\lambda}}H^0$ is
a Hilbert $A\rtimes_{\rho}H\rtimes_{\widehat{\rho}}H^0 -B\rtimes_{\sigma}H\rtimes_{\widehat{\sigma}}H^0$-bimodule
of finite type and $X\otimes M_N (\BC)$ is an exterior tensor product of a Hilbert $A-B$-bimodule $X$ of finite type
and an $M_N (\BC)-M_N (\BC)$-equivalence bimodule $M_N (\BC)$. Furthermore, there are unitary elements
$U\in (A\rtimes_{\rho}H\rtimes_{\widehat{\rho}}H^0 )\otimes H^0$ and
$V\in (B\rtimes_{\sigma}H\rtimes_{\widehat{\sigma}}H^0 )\otimes H^0$
such that
$$
U\widehat{\widehat{\lambda}}(x)V^* =((\Psi_X \otimes\id)\circ(\lambda\otimes\id_{M_N (\BC)})\circ\Psi_X^{-1})(x)
$$
for any $x\in X\rtimes_{\lambda}H\rtimes_{\widehat{\lambda}}H^0$.
\end{thm}
\begin{proof}
(1) Let $\Psi_X$ be as in Lemma \ref{lem:module}. By Lemmas \ref {lem:module} and \ref{lem:onto2},
we can see that $\Psi_X$
is a linear isomorphism from $X\otimes M_N (\BC)$
onto $X\rtimes_{\lambda}H\rtimes_{\widehat{\lambda}}H^0$. By \cite [Theorem 3.3]{KT2:coaction},
there are $U$ and $V$, unitary elements in
$(A\rtimes_{\rho, u}H\rtimes_{\widehat{\rho}}H^0 )\otimes H^0$ and
$(B\rtimes_{\sigma, v}H\rtimes_{\widehat{\sigma}}H^0 )\otimes H^0$ such that
\begin{align*}
\Ad(U)\circ\widehat{\widehat{\rho}} & =(\Psi_A \otimes\id)\circ(\rho\otimes\id_{M_N (\BC)})\circ\Psi_A^{-1}, \\
\Ad(V)\circ\widehat{\widehat{\sigma}} & =(\Psi_B \otimes\id)\circ(\sigma\otimes\id_{M_N (\BC)})\circ\Psi_B^{-1},
\end{align*}
respectively. Let $V^{\rho}$ and $V^{\sigma}$ be as above.
For any $\sum_{I, J}x_{IJ}\otimes f_{IJ}\in X\otimes M_N (\BC)$,
\begin{align*}
& U\widehat{\widehat{\lambda}}(\Psi_X (\sum_{I, J}x_{IJ}\otimes f_{IJ}))V^* \\
& =\sum_{I, J}(V_I^{\rho*}\otimes1^0 )V^{\rho}\
\widehat{\widehat{\lambda}}((1\rtimes_{\rho, u}1\rtimes_{\widehat{\rho}}\tau)
(x_{IJ}\rtimes_{\lambda}1\rtimes_{\widehat{\lambda}}1^0 )(1\rtimes_{\sigma, v}1\rtimes_{\widehat{\sigma}}\tau)) \\
& \times V^{\sigma*}(V_J^{\sigma}\otimes 1^0 )
\end{align*}
by \cite [Lemma 3.1]{KT2:coaction} since
$$
U=\sum_I (V_I^{\rho*}\otimes 1^0 )V^{\rho}\widehat{\widehat{\rho}}(V_I^{\rho}), \quad
V=\sum_I (V_I^{\sigma*}\otimes 1^0 )V^{\sigma}\widehat{\widehat{\sigma}}(V_I^{\sigma}).
$$
Since
\begin{align*}
\widehat{\widehat{\rho}}(1\rtimes_{\rho, u}1\rtimes_{\widehat{\rho}}\tau)
& =V^{\rho*}((1\rtimes_{\rho, u}1\rtimes_{\widehat{\rho}}\tau)\otimes 1^0 )V^{\rho}, \\
\widehat{\widehat{\sigma}}(1\rtimes_{\sigma, v}1\rtimes_{\widehat{\sigma}}\tau)
& =V^{\sigma*}((1\rtimes_{\sigma, v}1\rtimes_{\widehat{\sigma}}\tau)\otimes 1^0 )V^{\sigma}
\end{align*}
by the proof of \cite [Proposition 3.19]{KT1:inclusion},
\begin{align*}
& U\widehat{\widehat{\lambda}}(\Psi_X (\sum_{I, J}x_{IJ}\otimes f_{IJ}))V^* \\
& =\sum_{I, J}(V_I^{\rho*}\otimes 1^0 )V^{\rho}((x_{IJ}\rtimes_{\lambda}1\rtimes_{\widehat{\lambda}}1^0 )\otimes 1^0 )
V^{\sigma*}(V_J^{\sigma}\otimes 1^0 ) \\
& =\sum_{I, J}(V_I^{\rho*}\otimes 1^0 )\lambda(x_{IJ}\rtimes_{\lambda}1\rtimes_{\widehat{\lambda}}1^0 )
(V_J^{\sigma}\otimes1^0 )
\end{align*}
by Lemma \ref{lem:action}, where we identify $X$ with $X\rtimes_{\lambda}1$ and
$X\rtimes_{\lambda}1\rtimes_{\widehat{\lambda}}1^0$. On the other hand,
$$
((\Psi_X \otimes\id)\circ(\lambda\otimes\id))(x_{IJ}\otimes f_{IJ})
=(\Psi_X \otimes\id)(\lambda(x_{IJ})\otimes f_{IJ}) .
$$
We write that $\lambda(x_{IJ})=\sum_i y_{IJ \, i}\otimes\phi_i$, where
$\phi_i \in H^0 $, $y_{IJ \, i}\in X$ for any $I, J, i$.
Then
\begin{align*}
((\Psi_X \otimes\id)\circ & (\lambda\otimes\id))(x_{IJ}\otimes f_{IJ})
=\sum_{I, J, i}V_I^{\rho*}(y_{I J \, i}\rtimes_{\lambda}1\rtimes_{\widehat{\lambda}}1^0 )V_J^{\sigma}\otimes\phi_i \\
& =\sum_{I, J}(V_I^{\rho*}\otimes 1^0 )\lambda(x_{IJ}\rtimes_{\lambda}\rtimes_{\widehat{\lambda}}1^0 )
(V_J^{\sigma}\otimes 1^0 ).
\end{align*}
Therefore, we obtain the conclusion.
\newline
(2) We can prove (2) in the same way as (1).
\end{proof}

\section{The strong Morita equivalence for coactions and the Rohlin property}
\label{sec:rohlin}
For a unital $C^*$-algebra $A$, we set
\begin{align*}
c_0 (A) & =\{ \, (a_n )\in l^{\infty}(\BN, A) \, | \, \lim_{n\to\infty}||a_n ||=0 \, \}, \\
A^{\infty} & =l^{\infty} (\BN, A)/c_0 (A) .
\end{align*}
We denote an element in $A^{\infty}$ by the symbol $[a_n ]$
for an element $(a_n )\in l^{\infty}(\BN, A)$.
We identify $A$ with the $C^*$-subalgebra of $A^{\infty}$ consisting of the equivalence
classes of constant sequences and set
$$
A_{\infty}=A^{\infty}\cap A'.
$$
Let $X$ be a Hilbert $A-B$- bimodule of finite type, where $B$ is a unital $C^*$-algebra.
We define $X^{\infty}$ in the same way as above. We set 
\begin{align*}
c_0 (X) & =\{ \, (x_n )\in l^{\infty}(\BN, X) \, | \, \lim_{n\to\infty}||x_n ||=0 \, \}, \\
X^{\infty} & =l^{\infty} (\BN, X)/c_0 (X) .
\end{align*}
We denote an element in $X^{\infty}$ by the symbol $[x_n ]$
for an element $(x_n )\in l^{\infty}(\BN, X)$.
We regard $X^{\infty}$ as an $A^{\infty}-B^{\infty}$-bimodule as
follows: For any $[a_n ]\in A^{\infty}$, $[b_n ]\in B^{\infty}$, $[x_n ]\in X^{\infty}$,
$$
[a_n ][x_n ]=[a_n x_n ],  \quad [x_n ][b_n ]=[x_n b_n ] .
$$
Also, we define the left $A^{\infty}$-valued inner product and the right $B^{\infty}$-valued inner product
as follows: For any $[x_n ], [y_n ]\in X^{\infty}$,
$$
{}_{A^{\infty}} \la [x_n ] \, , \, [y_n ] \ra=[{}_A \la x_n \, , \, y_n \ra], \quad
\la [x_n ] \, , \, [y_n ] \ra_{B^{\infty}}=[ \la x_n \, , \, y_n \ra_B ] .
$$
By \cite [Lemma 2.5]{RW:continuous} and easy computations,
the above definitions are well-defined.
We identify $X$ with the Hilbert $A^{\infty}-B^{\infty}$-subbimodule of $X^{\infty}$ consisting of the equivalence
classes of constant sequences.
Also, we can see that $X^{\infty}$ is a complex vector
space satisfying Conditions (1)-(8) in \cite [Lemma 1.3]{KW2:discrete}.
Since $X$ is of finite type, there are finite subsets $\{u_i \}_{i=1}^n$, $\{v_j \}_{j=1}^m \subset X$
such that for any $x\in X$,
$$
\sum_{i=1}^n u_i \la u_i , x \ra_B =x=\sum_{j=1}^m {}_A \la x, v_j \ra v_j .
$$
Then we can regard $u_i $, $v_j \in X$ as elements in $X^{\infty}$ for $i=1, \dots, n$, $j=1, \dots, m$.
Thus $X^{\infty}$ is a Hilbert $A^{\infty}-B^{\infty}$-bimodule of finite type
by \cite [Lemma 1.3]{KW2:discrete}. Furthermore, if $X$ is an $A-B$-equivalence bimodule,
then $X^{\infty}$ is an $A^{\infty}-B^{\infty}$-equivalence bimodule.

\begin{lem}\label{lem:faithful}With the above notations, we suppose that $X$ is an
$A-B$-equivalence bimodule. Let $b\in B^{\infty}$. If $xb=0$ for any $x\in X$, then $b=0$,
where we regard $X$ as the Hilbert $A^{\infty}-B^{\infty}$-subbimodule of $X^{\infty}$.
\end{lem}
\begin{proof}Since $b\in B^{\infty}$, we write that $b=[b_m ]$, where $b_m \in B$ for any $m\in\BN$.
Since $xb=0$, $||xb_m ||\to 0$ $(m\to \infty)$. For any $y\in X$,
$$
|| \la y , x \ra_B \, b_m ||=|| \la y , xb_m \ra_B || \le ||y||\,||xb_m ||\to 0 \, (m\to\infty)
$$
by  \cite [Lemma 2.5]{RW:continuous}. On the other hand, there are
$x_1, \dots,x_n, y_1,\dots,y_n $ $\in X$ such that $\sum_{i=1}^n \la y_i , x_i \ra_B =1$ since $X$ is full
with the right $B$-valued inner product.
Hence
$$
||b_m ||=||\sum_{i=1}^n \la y_i , x_i  \ra_B \, b_m ||\le\sum_{i=1}^n || \la y_i , x_i \ra_B \, b_m ||\to 0.
$$
Therefore $b=0$.
\end{proof}

We are in position to present the main result in this paper.
Before doing it, we give the definitions of the approximate representability and the Rohlin property for
a coaction of a finite dimensional $C^*$-Hopf algebra on a unital $C^*$-algebra
and a remark on the definitions.

\begin{defin}\label{def:representable}(Cf. \cite [Definitions 4.3 and 5.1]{KT2:coaction})
Let $(\rho, u)$ be a twisted coaction
of a finite dimensional $C^*$-Hopf algebra $H^0$ on a unital $C^*$-algebra $A$. We say that $(\rho, u)$ is
\it approximately representable
\rm
if there is a unitary element $w\in A^{\infty}\otimes H^0$ satisfying the following conditions:
\newline
(1) $\rho(a)=(\Ad(w)\circ\rho_{H^0}^A )(a)$ for any $a\in A$,
\newline
(2) $u=(w\otimes 1^0 )(\rho_{H^0}^{A^{\infty}}\otimes\id)(w)(\id\otimes \Delta^0 )(w^* )$,
\newline
(3) $u=(\rho^{\infty}\otimes\id)(w)(w\otimes 1^0 )(\id\otimes\Delta^0 )(w^* )$.
\newline
Also, we say that
$(\rho, u)$ has
\it the Rohlin property
\rm
if its dual coaction $\widehat{\rho}$ of $H$ on $A\rtimes_{\rho}H$ is approximately representable.
\end{defin}

By \cite [Corollary 6.4]{KT2:coaction}, we can see that a coaction $\rho$ of $H^0$ on $A$ has the Rohlin property
if and only if there is a projection $p\in A_{\infty}$ such that $e\cdot_{\rho^{\infty}} p =\frac{1}{N}$, where
$N=\dim(H)$.

\begin{thm}\label{thm:preserve}Let $H$ be a finite dimensional $C^*$-Hopf algebra with its
dual $C^*$-Hopf algebra $H^0$. Let $\rho$ and $\sigma$ be coactions of $H^0$ on unital C$^*$-algebras $A$ and $B$,
respectively. We suppose that $\rho$ is strongly Morita equivalent to
$\sigma$. Then $\rho$ has the Rohlin property if and only if $\sigma$ has the Rohlin property.
\end{thm}
\begin{proof}Since $\rho$ and $\sigma$ are strongly Morita equivalent, there are an
$A-B$-equivalence bimodule $X$ and a coaction $\lambda$ of $H^0$ on $X$ with respect to $(A, B, \rho, \sigma)$.
According to the proof of Rieffel \cite [Proposition 2.1]{Rieffel:rotation}, we obtain the
following: Since $X$ is full with the right $B$-valued inner product,
there are elements $x_1 ,\dots,x_n, \, y_1,\dots,y_n \in X$ such that
$\sum_{i=1}^n \la x_i ,y_i \ra_B =1$. Let $E=A\otimes M_n (\BC)$ and we consider $X^n$ as an
$E-B$-equivalence bimodule in the usual way. Let $x=(x_i )_{i=1}^n$, $y=(y_i )_{i=i}^n \in X^n$.
Let $z={}_E \la y, y \ra^{\frac{1}{2}}x$ and let $q={}_E \la z, z \ra\in E$.
Then $q$ is a projection in $E$. Let $\pi$ be a map from $B$ to $E$ defined by $\pi(b)={}_E \la zb, z \ra$
for any $b\in B$. Then $\pi$ is an isomorphism of $B$ onto $qEq$. We suppose that
$\rho$ has the Rohlin property. Then by \cite [Corollary 6.4]{KT2:coaction} there is a projection $p\in A_{\infty}$
such that $e\cdot_{\rho^{\infty}}p=\frac{1}{N}$. We regard
$(X^{\infty})^n$ as an $E^{\infty}-B^{\infty}$-equivalence bimodule in the usual way.
Since $p\otimes I_n \in E^{\infty}$, there are elements
$$
u_1, \dots , u_m , \, v_1 , \dots ,v_m \in (X^{\infty})^n
$$
such that $p\otimes I_n =\sum_{k=1}^m {}_{E^{\infty}} \la u_k , v_k \ra$.
We write that
$$
u_k =(u_{k1}, \dots, u_{kn}), \quad v_k =(v_{k1}, \dots, v_{kn}),
$$
where $u_{ki}, \,  v_{ki}\in X^{\infty}$ for $k=1,2,\dots,m$, $i=1,2,\dots,n$.
Thus
$$
p\otimes I_n =\sum_{k=1}^m [ {}_{A^{\infty}} \la u_{ki} \, , \, v_{kj} \ra ]_{i,j=1}^n .
$$
Hence
$$
(***) \quad \sum_{k=1}^m {}_{A^{\infty}} \la u_{ki} \, , \, v_{kj} \ra= \begin{cases} p & i=j \\
0 & i\ne j \end{cases}.
$$
We note that since $p\in A_{\infty}$, $q(p\otimes I_n )q=q(p\otimes I_n )\in (qM_n (A)q)^{\infty}\cap (qM_n (A)q)'$.
Let $\pi^{\infty}$ be the isomorphism of $B^{\infty}$ onto $(qM_n (A)q)^{\infty}$
induced by $\pi$. Let $p_1 =(\pi ^{\infty})^{-1}(q(p\otimes I_n )q)$. Then $p_1$ is a projection in $B_{\infty}$
since $\pi(B)=qM_n (A)q$.
We show that $e\cdot_{\sigma^{\infty}}p_1
=\frac{1}{N}$. Since $q={}_E \la z, z \ra$,
\begin{align*}
q(p\otimes I_n )q & =\sum_{k=1}^m {}_{E^{\infty}}\la {}_E \la z, z \ra u_k \, , \, {}_E \la z, z \ra v_k \ra \\
& =\sum_{k=1}^m {}_{E^{\infty}}\la z \la z, u_k \ra_{B^{\infty}}\, \la v_k , z \ra_{B^{\infty}}\, , \, z \ra \\
& =\pi^{\infty}(\sum_{k=1}^m \la z, u_k \ra_{B^{\infty}}\, \la v_k , z \ra_{B^{\infty}}).
\end{align*}
Thus
$$
p_1 =\sum_{k=1}^m \la z, u_k \ra_{B^{\infty}} \, \la v_k , z \ra_{B^{\infty}}
=\sum_{k=1}^m \la z \, , \, {}_{E^{\infty}} \la u_k \, , \, v_k \ra z \ra_{B^{\infty}} .
$$
Since $z\in X^n$, we write $z=(z_i )_{i=1}^n $, where $z_i \in X$ for
$i=1,2,\dots,n$. Hence by Equation ($***$),
\begin{align*}
p_1 & = \la \left[
\begin{array}{ccc}
z_1 \\
\vdots\\
z_n 
\end{array}
\right] \, , \, \sum_{k=1}^m \left[\,{}_{A^{\infty}} \la u_{ki} \, , \, v_{kj} \ra \right]_{i, j=1}^n
\left[
\begin{array}{ccc}
z_1 \\
\vdots\\
z_n 
\end{array}
\right] \ra_{B^{\infty}} \\
& =\sum_{i, j=1}^n \la z_i \, , \, \sum_{k=1}^m {}_{A^{\infty}} \la u_{ki} \, , \, v_{kj} \ra z_j \ra_{B^{\infty}} \\
& =\sum_{i=1}^n \la z_i \, , \, pz_i \ra_{B^{\infty}} .
\end{align*}
For any $w\in X$,
\begin{align*}
w[e\cdot_{\sigma^{\infty}}p_1 ] & =\sum_{i=1}^n w \la [S(e_{(1)}^* )\cdot_{\lambda}z_i ]\, , \, [e_{(2)}
\cdot_{\lambda^{\infty}}pz_i ] \ra_{B^{\infty}} \\
& =\sum_{i=1}^n {}_A \la w\, , \, [S(e_{(1)}^* )\cdot_{\lambda}z_i ] \ra [e_{(2)}\cdot_{\lambda^{\infty}}pz_i ] \\
& =\sum_{i=1}^n {}_A \la [e_{(2)}S(e_{(1)})\cdot_{\lambda}w ] \, , \, [S(e_{(3)}^* )\cdot_{\lambda}z_i ] \ra \,
[e_{(4)}\cdot_{\lambda^{\infty}}pz_i ] \\
& =\sum_{i=1}^n [e_{(2)}\cdot_{\rho} {}_A \la [S(e_{(1)})\cdot_{\lambda}w]\, , \, z_i \ra ]
[e_{(3)}\cdot_{\lambda^{\infty}}pz_i ] \\
& =\sum_{i=1}^n [e_{(2)}\cdot_{\lambda^{\infty}}p[S(e_{(1)})\cdot_{\lambda}w]\la z_i , z_i \ra_B ] \\
& =[e_{(2)}\cdot_{\rho^{\infty}}p][e_{(3)}S(e_{(1)})\cdot_{\lambda}w].
\end{align*}
Since $e=\sum_{i,k}\frac{d_k}{N}w_{ii}^k$,
\begin{align*}
w[e\cdot_{\sigma^{\infty}}p_1 ]& =\sum_{i, j, k, j_1}\frac{d_k}{N}[w_{jj_1}^k \cdot_{\rho^{\infty}}p]
[w_{j_1 i}^k S(w_{ij}^k )\cdot_{\lambda}w ] \\
& =\sum_{j, k}\frac{d_k}{N}[w_{jj}^k \cdot_{\rho^{\infty}}p]w=[e\cdot_{\rho^{\infty}}p]w=\frac{1}{N}w.
\end{align*}
Thus $e\cdot_{\sigma^{\infty}}p_1 =\frac{1}{N}$ by Lemma \ref{lem:faithful}. Therefore we obtain the conclusion
by \cite [Corollary 6.4.]{KT2:coaction}.
\end{proof}

\begin{cor}\label{cor:twisted}Let $(\rho, u)$ and $(\sigma, v)$ be twisted coactions of
$H^0$ on $A$ and $B$, respectively. We suppose that they are strongly Morita equivalent.
Then the following hold:
\newline
$(1)$ The twisted coaction $(\rho, u)$ has the Rohlin property if and only if so does $(\sigma, v)$,
\newline
$(2)$ The twisted coaction $(\rho, u)$ is approximately representable if and only if
so is $(\sigma, v)$.
\end{cor}
\begin{proof}(1) We suppose that $(\rho, u)$ has the Rohlin property. Then $\widehat{\widehat{\rho}}$ has
the Rohlin property by \cite [Proposition 5.5]{KT2:coaction}. Also, since $(\rho, u)$ and $(\sigma, v)$
are strongly Morita equivalent, $\widehat{\widehat{\rho}}$ and $\widehat{\widehat{\sigma}}$ are
strongly Morita equivalent by Corollary \ref{cor:dual}.
Thus $(\sigma, v)$ has the Rohlin property by Theorem \ref{thm:preserve}
and \cite [Proposition 5.5]{KT2:coaction}.
\newline
(2) We suppose that $(\rho, u)$ is approximately representable. Then $\widehat{\rho}$ has the
Rohlin property by the definition of the Rohlin property and \cite [Proposition 4.6]{KT2:coaction}.
Since $(\rho, u)$ and $(\sigma, v)$ are strongly Morita equivalent,
$\widehat{\rho}$ and $\widehat{\sigma}$ are strongly Morita equivalent by Corollary \ref{cor:dual}.
Thus by Theorem \ref{thm:preserve}, $\widehat{\sigma}$ has the Rohlin property.
Hence by the definition of the Rohlin property and \cite [Proposition 4.6]{KT2:coaction}, $(\sigma, v)$ is
approximately representable.
\end{proof}

\section{Application}\label{sec:application}
Let $A$ and $B$ be unital $C^*$-algebras and $H$ a finite dimensional $C^*$-Hopf algebra
with its dual $C^*$-Hopf algebra $H^0$.
We suppose that $A$ is strongly Morita equivalent to $B$. Let $\rho$ be a coaction of $H^0$ on $A$.
By \cite [Proposition 2.1]{Rieffel:rotation}, there are an $n\in \BN$ and a full projection $q\in M_n (A)$
such that $B$ is isomorphic to $qM_n (A)q$. We identify $B$ with $qM_n (A)q$.
We suppose that $(\rho\otimes\id)(q)\sim q\otimes 1^0$ in $M_n (A)\otimes H^0$. Hence there is a partial isometry
$w\in M_n (A)\otimes H^0$ such that $w^* w =(\rho\otimes\id)(q)$, $ww^* =q\otimes 1^0$.

\begin{lem}\label{lem:partial}With the above notations, there is a partial isometry
$z\in M_n (A)\otimes H^0$ such that $z^* z=(\rho\otimes\id)(q)$, $zz^* =q\otimes 1^0$ and that $\widehat{z}(1)=q$.
\end{lem}
\begin{proof}We note that $\widehat{w^* }(1)=\widehat{w}(1)^*$.
Since $w^* w=(\rho\otimes\id)(q)$ and $ww^* =q\otimes 1^0$,
$$
\widehat{w^*}(1)\widehat{w}(1)=(\id\otimes\epsilon^0 )((\rho\otimes\id)(q))=q, \quad
\widehat{w}(1)\widehat{w^*}(1)=q.
$$
Let $z=(\widehat{w^* }(1)\otimes 1^0 )w$. Then $\widehat{z}(1)=\widehat{w^* }(1)\widehat{w}(1)=q$.
Also,
\begin{align*}
z^* z & =w^* (\widehat{w}(1)\otimes 1^0 )(\widehat{w^* }(1)\otimes 1^0 )w=(\rho\otimes\id)(q), \\
zz^* & =(\widehat{w^* }(1)\otimes 1^0 )ww^* (\widehat{w}(1)\otimes 1^0 )=q\otimes 1^0 .
\end{align*}
Therefore, we obtain the conclusion. 
\end{proof}

Let
\begin{align*}
\sigma & =\Ad(z)\circ(\rho\otimes\id_{M_n (\BC)}), \\
u & =(z\otimes 1^0 )(\rho\otimes\id_{M_n (\BC)}\otimes\id_{H^0})(z)(\id_{M_n (A)}\otimes\Delta^0 )(z^* ) .
\end{align*}
We note that $u\in B\otimes H^0 \otimes H^0$.
We shall show that $(\sigma, u)$ is a twisted coaction of $H^0$ on $B$, which is strongly Morita equivalent to
$\rho$. We sometimes identify $A\otimes H^0 \otimes M_n (\BC)$ with $A\otimes M_n (\BC) \otimes H^0$.

\begin{lem}\label{lem:weak}With the above notations, $\sigma$ is a weak coaction of $H^0$ on $B$.
\end{lem}
\begin{proof}For any $x\in M_n (A)$,
$$
\sigma(qxq) =z(\rho\otimes\id)(qxq)z^*
=(q\otimes 1^0 )z(\rho\otimes\id)(x)z^* (q\otimes1^0 ) .
$$
Hence $\sigma$ is a map from $B$ to $B\otimes H^0$. Also, by routine computations,
we can see that $\sigma$ is a homomorphism of $B$ to
$B\otimes H^0$ with $\sigma(q)=q\otimes 1^0$. Furthermore, since $\widehat{z}(1)=q$, for any $x\in M_n (A)$,
\begin{align*}
(\id\otimes\epsilon^0 )(\sigma(qxq)) & =(\id\otimes\epsilon^0 )((q\otimes 1^0 )z(\rho\otimes\id)(x)z^* (q\otimes 1^0 )) \\
& =q\widehat{z}(1)(\id\otimes\epsilon^0 )((\rho\otimes\id)(x))\widehat{z^* }(1)q=qxq.
\end{align*}
Thus $\sigma$ is a weak coaction of $H^0$ on $B$.
\end{proof}

\begin{lem}\label{lem:twisted}With the above notations, $(\sigma, u)$ is a twisted coaction
of $H^0$ on $B$.
\end{lem}
\begin{proof}By routine computations, we can see that $uu^* =u^* u=q\otimes 1^0 \otimes 1^0$. Thus
$u$ is a unitary element in $B\otimes H^0 \otimes H^0$. For any $x\in M_n (A)$,
\begin{align*}
& ((\sigma\otimes\id_{H^0})\circ\sigma)(qxq) \\
& =(z\otimes 1^0 )(\rho\otimes\id\otimes\id_{H^0})(z)((\rho\otimes\id\otimes\id_{H^0})\circ(\rho\otimes\id))(qxq) \\
& \times (\rho\otimes\id\otimes\id_{H^0})(z^* )(z^* \otimes 1^0 ).
\end{align*}
On the other hand,
\begin{align*}
& (\Ad(u)\circ(\id\otimes\Delta^0 )\circ\sigma)(qxq) \\
& =(z\otimes 1^0 )(\rho\otimes\id\otimes\id_{H^0})(z)((\id\otimes\Delta^0 )\circ(\rho\otimes\id))(qxq) \\
& \times (\rho\otimes\id\otimes\id_{H^0})(z^* )(z^* \otimes 1^0 ).
\end{align*}
Since $(\rho\otimes\id\otimes\id_{H^0})\circ(\rho\otimes\id)=(\id\otimes\Delta^0 )\circ(\rho\otimes\id)$,
we obtain that
$$
(\sigma\otimes\id_{H^0} )\circ\sigma=\Ad(u)\circ(\id\otimes\Delta^0 )\circ\sigma.
$$
Also,
\begin{align*}
& (u\otimes 1^0 )(\id\otimes\Delta^0 \otimes\id_{H^0})(u) \\
& =(z\otimes 1^0 \otimes 1^0 )(\rho\otimes\id\otimes\id_{H^0}\otimes\id_{H^0})(z\otimes 1^0 ) \\
& \times (\id\otimes\Delta^0 \otimes\id_{H^0})((\rho\otimes\id\otimes\id_{H^0})(z)(\id\otimes\Delta^0 )(z^* )) .
\end{align*}
On the other hand, since $(\rho\otimes\id\otimes\id_{H^0})\circ(\rho\otimes\id)=(\id\otimes\Delta^0 )\circ(\rho\otimes\id)$,
\begin{align*}
& (\sigma\otimes\id_{H^0}\otimes\id_{H^0})(u)(\id\otimes\id_{H^0}\otimes\Delta^0 )(u) \\
& =(z\otimes 1^0 \otimes1^0 )(\rho\otimes\id\otimes\id_{H^0}\otimes\id_{H^0})(z\otimes 1^0 ) \\
& \times (\id\otimes\Delta^0 \otimes\id_{H^0})((\rho\otimes\id\otimes\id_{H^0})(z)) \\
& \times (\rho\otimes\id\otimes\id_{H^0}\otimes\id_{H^0})((\id\otimes\Delta^0 )(z^* )) \\
& \times(\id\otimes\id_{H^0}\otimes\Delta^0 )((\rho\otimes\id\otimes\id_{H^0})(z))
(\id\otimes\Delta^0 \otimes\id_{H^0})((\id\otimes\Delta^0 )(z^* )).
\end{align*}
We can see that
$$
(\rho\otimes\id\otimes\id_{H^0}\otimes\id_{H^0})\circ(\id\otimes\Delta^0 )
=(\id\otimes\id_{H^0}\otimes\Delta^0 )\circ(\rho\otimes\id\otimes\id_{H^0})
$$
by easy computations. Furthermore, we note that
\begin{align*}
& (\id\otimes\id_{H^0}\otimes\Delta^0 )\circ(\id\otimes\Delta^0 )\circ(\rho\otimes\id) \\
& =(\id\otimes\Delta^0 \otimes\id_{H^0})\circ(\id\otimes\Delta^0 )\circ(\rho\otimes\id) \\
& =(\id\otimes\Delta^0 \otimes\id_{H^0})\circ(\rho\otimes\id\otimes\id_{H^0})\circ(\rho\otimes\id).
\end{align*}
Thus since
\begin{align*}
& (\id\otimes\id_{H^0}\otimes\Delta^0 )((\rho\otimes\id\otimes\id_{H^0})((\rho\otimes\id)(q))) \\
& =(\id\otimes\id_{H^0}\otimes\Delta^0 )((\id\otimes\Delta^0 )((\rho\otimes\id)(q))) ,
\end{align*}
\begin{align*}
& (\sigma\otimes\id_{H^0}\otimes\id_{H^0})(u)(\id\otimes\id_{H^0}\otimes\Delta^0 )(u) \\
& =(z\otimes 1^0 \otimes1^0 )(\rho\otimes\id\otimes\id_{H^0}\otimes\id_{H^0})(z\otimes 1^0 ) \\
& \times (\id\otimes\Delta^0 \otimes\id_{H^0})((\rho\otimes\id\otimes\id_{H^0})(z)) \\
& \times(\id\otimes\id_{H^0}\otimes\Delta^0 )((\rho\otimes\id\otimes\id_{H^0})((\rho\otimes\id)(q))) \\
& \times (\id\otimes\Delta^0 \otimes\id_{H^0})((\id\otimes\Delta^0 )(z^* )) \\
& =(z\otimes 1^0 \otimes 1^0 )(\rho\otimes\id\otimes\id_{H^0}\otimes\id_{H^0})(z\otimes 1^0 ) \\
& \times (\id\otimes\Delta^0 \otimes\id_{H^0})((\rho\otimes\id\otimes\id_{H^0})(z)(\id\otimes\Delta^0 )(z^* )) .
\end{align*}
Hence we obtain that
$$
(u\otimes 1^0 )(\id\otimes\Delta^0 \otimes\id_{H^0})(u)
=(\sigma\otimes\id_{H^0}\otimes\id_{H^0})(u)(\id\otimes\id_{H^0}\otimes\Delta^0 )(u).
$$
Furthermore, since $\widehat{z}(1)=q$, for any $h\in H$,
\begin{align*}
(\id\otimes h\otimes\epsilon^0 )(u) & =\widehat{z}(h_{(1)})[h_{(2)}\cdot_{\rho\otimes\id}q]\widehat{z^* }(h_{(3)})
=(\id\otimes h)(\sigma(q))=\epsilon(h)q , \\
(\id\otimes\epsilon^0 \otimes h)(u) & =\widehat{z}(1)[1\cdot_{\rho\otimes\id}\widehat{z}(h_{(1)})]\widehat{z^*}(h_{(2)})
=\widehat{z}(1)\epsilon(h)=\epsilon(h)q.
\end{align*}
Therefore, $(\sigma, u)$ is a twisted coaction of $H^0$ on $B$.
\end{proof}

Let $f$ be a minimal projection in $M_n (\BC)$ and let $p$ be a full projection in $M_n (A)$
defined by $p=1_A \otimes f$. Let $X=pM_n (A)q$. We regard $X$ as an $A-B$-equivalence bimodule
in the usual way, where we identify $A$ and $B$ with $pM_n (A)p$ and $qM_n (A)q$, respectively.
Then we can regard $X$ as a set $\{[a_1 , \dots , a_n ]q \, | \, a_i \in A, \, i=1, \dots,n \}$.
Let $\lambda$ be a linear map from $X$ to $X\otimes H^0$ defined by
\begin{align*}
\lambda([a_1 , \dots , a_n ]q) & =[\rho(a_1 ), \dots , \rho(a_n )](\rho\otimes\id)(q)z^* \\
& =[\rho(a_1 ), \dots , \rho(a_n )]z^* (q\otimes 1^0 )
\end{align*}
for any $[a_1 , \dots, a_n ]q\in X$.

\begin{lem}\label{lem:coaction}With the above notations, $\lambda$ is a twisted coaction of $H^0$ on $X$
with respect to $(A, B, \rho, \sigma, u)$.
\end{lem}
\begin{proof}By routine computations, we can see that $\lambda$ is a weak coaction of $H^0$
on $X$ with respect to $(A, B, \rho, \sigma, u)$. For any $[a_1 , \dots, a_n ]q\in X$,
\begin{align*}
& ((\lambda\otimes\id_{H^0})\circ\lambda)([a_1, \dots, a_n ]q) \\
& =[((\rho\otimes\id_{H^0})\circ\rho)(a_1 ), \dots, ((\rho\otimes\id_{H^0})\circ\rho)(a_n )] \\
& \times (\rho\otimes\id\otimes\id_{H^0})(z^*)(z^* \otimes 1^0 ). \\
\end{align*}
On the other hand, since $(\rho\otimes\id_{H^0})\circ\rho=(\id\otimes\Delta^0 )\circ\rho$,
\begin{align*}
& ((\id\otimes\Delta^0 )\circ\lambda)([a_1 , \dots, a_n ]q)u^* \\
& =[((\id\otimes\Delta^0 )\circ\rho)(a_1 ), \dots, ((\id\otimes\Delta^0 )\circ\rho)(a_n )] \\
& \times (\rho\otimes\id\otimes\id_{H^0})(z^* )(z^* \otimes 1^0) .
\end{align*}
Hence for any $[a_1 , \dots, a_n ]q\in X$,
$$
((\lambda\otimes\id_{H^0})\circ\lambda)([a_1, \dots, a_n ]q)
=((\id\otimes\Delta^0 )\circ\lambda)([a_1 , \dots, a_n ]q)u^* .
$$
Thus $\lambda$ is a twisted coaction of $H^0$ on $X$
with respect to $(A, B, \rho, \sigma, u)$.
\end{proof}

\begin{thm}\label{thm:appli}Let $A$ be a unital $C^*$-algebra and $H$ a finite dimensional
$C^*$-Hopf algebra with its dual $C^*$-Hopf algebra $H^0$. Let $\rho$ be a coaction of $H^0$ on $A$
with the Rohlin property. Let $q$ be a full projection in a $C^*$-algebra $M_n (A)$ such that
$$
(\rho\otimes\id_{M_n (\BC)})(q)\sim q\otimes 1^0
$$
in $M_n (A)\otimes H^0$. Let $B=qM_n (A)q$. Then there is a coaction of $H^0$ on $B$ with
the Rohlin property.
\end{thm}
\begin{proof}By Lemmas \ref{lem:twisted} and \ref{lem:coaction}, there is a twisted coaction $(\sigma, u)$
such that $(\sigma, u)$ is strongly Morita equivalent to $\rho$. By Corollary \ref{cor:twisted}, $(\sigma, u)$ has the
Rohlin property. Furthermore, by \cite [Theorem 9.6]{KT2:coaction}, there is a unitary element $y\in B\otimes H^0$
such that
$$
(y\otimes 1^0 )(\sigma\otimes\id_{H^0})u(\id\otimes\Delta^0 )(y^* )=1_B \otimes 1^0 \otimes 1^0 .
$$
Let $\sigma_1 =\Ad(y)\circ\sigma$. Then $\sigma_1$ is a coaction of $H^0$ on $B$ with the
Rohlin property by easy computations since $\sigma_1$ is exterior equivalent to $(\sigma, u)$.
\end{proof}

Let $A$ be a UHF-algebra of type $N^{\infty}$, where $N$ is the dimension of a finite dimensional
$C^*$-Hopf algebra $H$.
In \cite{KT2:coaction}, we showed that there
is a coaction $\rho$ of $H^0$ on $A$ with the Rohlin property.
\begin{cor}\label{cor:UHF}With the above notations, for any unital $C^*$-algebra $B$, that is
strongly Morita equivalent to $A$, there is a coaction $\sigma$ of $H^0$ on $B$ with the Rohlin property.
\end{cor}
\begin{proof}By \cite [Proposition 2.1]{Rieffel:rotation} there are $n\in \BN$ and a full projection $q\in M_n (A)$
such that $B$ is isomorphic to $qM_n (A)q$. We identify $B$ with $qM_n (A)q$.
Let $\rho$ be a coaction of $H^0$ on $A$ with the Rohlin property. Then \cite [Lemma 10.10]{KT2:coaction},
$(\rho\otimes\id_{M_n (\BC)})(q)\sim q\otimes 1^0 $ in $M_n (A)\otimes H^0$ since $A$ has cancellation.
Therefore, by Theorem \ref{thm:appli} we obtain the conclusion.
\end{proof}

\sl
Acknowledgement.
\rm
The authors wish to thank the referee for valuable suggestions for improvement of
the manuscript.


\end{document}